\documentclass[reqno]{amsart}

\usepackage{amsmath,amsthm,amssymb}
\usepackage{mathtools}
\usepackage{tikz}
\usetikzlibrary{calc,shapes.geometric,cd}
\usepackage{multicol}
\usepackage{array}
\usepackage[colorlinks=false]{hyperref}
\usepackage{cite}
\usepackage{microtype}
\usepackage{paracol}  
\usepackage{graphicx}
\usepackage{mathrsfs}

\title{2-dimensional Shephard groups}
\author{Katherine M. Goldman}

\usepackage{thmtools}
\usepackage{thm-restate}

\newtheorem{thm}{Theorem}[section]
\newtheorem{prop}[thm]{Proposition}
\newtheorem{lemma}[thm]{Lemma}
\newtheorem{cor}[thm]{Corollary}
\newtheorem*{thm*}{Theorem}
\newtheorem*{lemma*}{Lemma}
\newtheorem*{cor*}{Corollary}

\theoremstyle{definition}
\newtheorem{defn}[thm]{Definition}
\newtheorem{notat}[thm]{Notation}

\newtheorem{ex}[thm]{Example}
\newtheorem*{ex*}{Example}

\newtheorem*{assumption}{Assumption}

\numberwithin{equation}{section}

\renewcommand{\c}{\mathcal}
\renewcommand{\ker}{\mathrm{ker}}
\newcommand{\im}{\mathrm{im}}
\newcommand{\cat}{\mathrm{CAT}}

\DeclareMathOperator{\lcm}{lcm}
\DeclareMathOperator{\isom}{Isom}
\newcommand{\bb}[2][]{\ensuremath{\mathbb{#2}^{#1}}}
\newcommand{\SL}{\mathrm{SL}}
\newcommand{\PSL}{\mathrm{PSL}}
\newcommand{\sh}{\mathrm{Sh}}
\newcommand{\Hom}{\mathsf{Hom}}
\newcommand{\id}{\mathrm{id}}
\newcommand{\scr}[1]{\mathscr{#1}}

\begin{document}

\begin{abstract}
    The 2-dimensional Shephard groups are quotients of 2-dimensional Artin groups by powers of standard generators. We show that such a quotient is not $\cat(0)$ if the powers taken are sufficiently large. However, for a given 2-dimensional Shephard group, we construct a $\cat(0)$ piecewise Euclidean cell complex with a cocompact action (analogous to the Deligne complex for an Artin group) that allows us to determine other non-positive curvature properties. Namely, we show the 2-dimensional Shephard groups are acylindrically hyperbolic (which was known for 2-dimensional Artin groups), and relatively hyperbolic (which most Artin groups are known not to be). As an application, we show that a broad class of 2-dimensional Artin groups are residually finite.
\end{abstract}

\maketitle 

\section{Introduction}

Shephard groups (named for G.C.~Shephard \cite{shep52}) are specific quotients of Artin groups and include, for example, the Coxeter groups and graph products of cyclic groups as special cases. 
In a previous paper \cite{goldman2023cat0}, the author identified a specific class of Shephard groups which exhibit Coxeter-like behavior, and proved that these are $\cat(0)$. However, the Shephard groups form a very broad class, and because of this, we can find examples which exhibit behavior that diverges quite a bit from the intuition that Coxeter groups and Artin groups provide. The motivating criteria in \cite{goldman2023cat0} can be roughly summarized by saying that the finite parabolic subgroups are identical to the finite parabolic subgroups of the ``associated'' Coxeter group. 
This class is rather inflexible, though; there are only so many finite Shephard groups that are not finite Coxeter groups. 
Our main motivation, then, is to begin the study of Shephard groups which do not satisfy this criteria. That is to say, to study those Shephard groups which possess some infinite parabolic subgroup whose associated Coxeter group is finite. 

It turns out that such groups are quite fascinating even in the 2-generator (or ``dihedral'') case, and possess a rather interesting geometry which deviates from the Artin and Coxeter groups. 
These dihedral groups will be one of the main focuses of the current article. The other main focus are the 2-dimensional Shephard groups. These are the Shephard groups which can be reasonably said to be ``built-up'' out of their dihedral subgroups.  
From our results on 2-generator Shephard groups, we derive interesting information about the geometry of the 2-dimensional Shephard groups, some known for Artin groups, some different than Artin groups, and some which can be used to show new properties of Artin groups.
The full definitions and statements are as follows. 

Let $\Gamma$ be a simplicial graph with vertices $i$ labeled by $p_i \in \bb{Z}_{\geq 2} \cup \{\infty\}$ and edges $\{i,j\}$ labeled by $m_{ij} \in \bb{Z}_{\geq 2}$. If $m_{ij}$ is odd then we require $p_i = p_j$. We call $\Gamma$ an \emph{extended} (or \emph{Shephard}) \emph{presentation graph}. The Shephard group $\sh_\Gamma$ with presentation graph $\Gamma$ is 
\begin{equation*} 
    \sh_\Gamma = \left\langle
    V(\Gamma) \ \middle| \ \begin{matrix}
        \mathrm{prod}(i,j; m_{ij}) = \mathrm{prod}(j,i; m_{ij}) \text{ if } \{i,j\} \in E(\Gamma) \\
        i^{p_i} = 1 \text{ if } p_i < \infty
    \end{matrix}
     \right\rangle,
\end{equation*}
where $\mathrm{prod}(a,b; m)$ denotes $(ab)^{m/2}$ if $m$ is even and $(ab)^{(m-1)/2}a$ if $m$ is odd.
If $\Gamma = \begin{tikzpicture}[baseline=-0.55ex, scale=0.75, every node/.style={scale=0.75}]
        \filldraw (0,0) circle (0.05cm);
        \draw (0,0) -- node[above] {$q$} (1,0);
        \filldraw (1,0) circle (0.05cm);
        \node at (-0.3,0) {$p$};
        \node at (1.3,0) {$r$};
\end{tikzpicture}$ is a single edge, then we will write $\sh_\Gamma = \sh(p,q,r)$ and call this an \emph{edge (Shephard) group} or a \emph{dihedral Shephard group}. 

If $\Gamma$ is an extended presentation graph and $p \in \bb{Z}_{\geq 2} \cup \{\infty\}$, we denote by $\Gamma(p)$ the extended presentation graph obtained from $\Gamma$ by replacing all $p_i$ with $p$.
In particular, for any extended presentation graph $\Gamma$, we define
\begin{align*}
    W_\Gamma &= \sh_{\Gamma(2)} \\
    A_\Gamma &= \sh_{\Gamma(\infty)},
\end{align*}
which are the Coxeter group and Artin group, resp., associated to $\Gamma$. Thus we see that every Shephard group is a quotient of an Artin group by some powers of standard generators. 

\begin{assumption}
     In the rest of the paper, ``extended presentation graph'' will be taken to mean ``extended presentation graph with all $p_i$ finite'' (so the kernel of $A_\Gamma \to \sh_\Gamma$ contains powers of every generator), unless we explicitly mention $\Gamma(\infty)$, and ``Shephard group'' will mean ``Shephard group with all finite-order generators''.
\end{assumption}

We write $\Lambda \leq \Gamma$ if $\Lambda$ is a full subgraph of $\Gamma$ and inherits the edge and vertex labels of $\Gamma$. 
(A full or ``induced'' subgraph $\Lambda$ of $\Gamma$ is one where if $v,w \in V(\Lambda)$ and $\{v,w\} \in E(\Gamma)$, then $\{v,w\} \in E(\Lambda)$.)  
Then $\Lambda$ is also an extended presentation graph and thus determines a Shephard group $\sh_\Lambda$, a Coxeter group $W_\Lambda$, and an Artin group $A_\Lambda$. If $W_\Lambda$ is finite, then we will call each of $\Lambda$, $\sh_\Lambda$, and $A_\Lambda$ ``spherical-type''.
(Similarly, if $W_\Gamma$ is word hyperbolic, then we will call each of $\Gamma$, $\sh_\Gamma$, and $A_\Gamma$ ``hyperbolic-type''.)
A property of an extended presentation graph $\Gamma$ commonly of interest for Artin groups and Coxeter groups is
\begin{enumerate}
    \item[(2D)] For all spherical-type $\Lambda \leq \Gamma$, $|V(\Lambda)| \leq 2$.
\end{enumerate}
If a graph $\Gamma$ satisfies (2D), we will call $\Gamma$ \emph{2-dimensional}. 
Much of the behavior of 2-dimensional Shephard groups can be deduced from their dihedral subgroups. Our first main result concerns solely these dihedral Shephard groups.

\begin{restatable}{theorem}{maindihedral} \label{thm:main2gen}
    Let $(p,q,r)$ be a triple of integers each $\geq 2$, with $p = r$ if $q$ is odd, and let $h = 1/p + 2/q + 1/r$. If $h \leq 1$, then $\sh(p,q,r)$ cannot admit a proper action by semi-simple isometries on any $\cat(0)$ space. In particular, $\sh(p,q,r)$ is not $\cat(0)$. Moreover,
    \begin{enumerate}
        \item If $h = 1$, $\sh(p,q,r)$ is commensurable to the 3-dimensional integer Heisenberg group; in particular, it is virtually nilpotent and not semihyperbolic.
        \item If $h < 1$, $\sh(p,q,r)$ is commensurable to the ``universal central extension'' of a hyperbolic surface group, is a uniform lattice in $\isom(\widetilde{\SL_2\bb{R}})$, and is biautomatic.
    \end{enumerate}
    In particular, $\sh(p,q,r)$ is linear, and if an element of $\sh(p,q,r)$ has finite order, it is conjugate to a power of one of the standard generators.
\end{restatable}

Not mentioned is the case $1/p + 2/q + 1/r > 1$; this is because if this holds, then $\sh(p,q,r)$ is finite (and is handled in \cite{goldman2023cat0}). In other words, $1/p + 2/q + 1/r \leq 1$ if and only if $\sh(p,q,r)$ is infinite. 
The fact that finite dihedral Shephard groups are linear is shown in \cite{1975regular}; hence, the above Theorem also implies all dihedral Shephard groups are linear.

\begin{ex*}
    The braid group on three strands $\sh(\infty,3,\infty)$ is known to be $\cat(0)$ (e.g., in \cite{brady2010braids}), but its quotient $\sh(p,3,p)$ is not $\cat(0)$ when $p \geq 6$ (when $p < 6$, this quotient is finite).
\end{ex*}

The broadest class of Shephard groups to which Theorem \ref{thm:main2gen} applies are described in the following.

\begin{cor*}
    Let $\Gamma$ be an extended presentation graph (possibly with infinite vertex labels) with an edge $e = \{i,j\}$ such that 
    \begin{enumerate}
        \item the parabolic $\sh_e$ generated by $e$ embeds\footnote{Meaning $\sh_e$ is isomorphic to the subgroup $\langle V(e) \rangle$ of $\sh_\Gamma$ via the map induced by the inclusion $e \hookrightarrow \Gamma$ on generators.} in $\sh_\Gamma$,
        \item the edge and vertex labels of $e$ are finite, and
        \item $\sh_e$ is infinite
    \end{enumerate}
    Then $\sh_\Gamma$ cannot admit a proper action on a $\cat(0)$ space by semi-simple isometries. In particular, $\sh_\Gamma$ is not $\cat(0)$. Moreover, $1/p_i + 2/m_{ij} + 1/p_j = 1$, then $\sh_\Gamma$ is not semihyperbolic.
\end{cor*} 

To rephrase this corollary, an embedded infinite edge subgroup with all labels finite is a ``poison subgroup'' for being $\cat(0)$, or even semihyperbolic in some cases.
It is a conjecture that given any extended presentation graph, its edge parabolics embed, so, conjecturally, (1) is unnecessary and infinite edge groups with finite labels are always poison subgroups.
But, as a consequence of the upcoming Theorem \ref{thm:cocptaction}, the edge groups of a 2-dimensional presentation graph always embed, so we can say

\begin{cor*}
    Let $\Gamma$ be a 2-dimensional extended presentation graph with an edge $e = \{i,j\}$ such that $p_i < \infty$, $p_j < \infty$, $m_{ij} < \infty$, and $\frac{1}{p_i} + \frac{2}{m_{ij}} + \frac{1}{p_j} \leq 1$. Then $\sh_\Gamma$ cannot admit a proper action by semi-simple isometries on a $\cat(0)$ space, and in particular is not $\cat(0)$. Moreover, if $\frac{1}{p_i} + \frac{2}{m_{ij}} + \frac{1}{p_j} = 1$, then $\sh_\Gamma$ is not semihyperbolic.
\end{cor*}

\begin{ex*}
    Suppose $\Gamma$ is an XXL-type presentation graph (meaning all edge labels are $\geq 5$); in particular $\Gamma$ is 2-dimensional. Then $A_\Gamma$ is $\cat(0)$ \cite{haettel2022xxl}. However there is some $N$ such that the quotient $\sh_{\Gamma(p)}$ is not $\cat(0)$ for all $p \geq N$.
\end{ex*}

In light of these examples, Theorem \ref{thm:main2gen} may seem perplexing. 
Perhaps it is less surprising when one considers that the $\cat(0)$ condition requires non-positive curvature at all scales, and we are able to show that (most of) these groups possess other, less restrictive non-positive curvature properties. 

\begin{restatable}{theorem}{cocptaction} \label{thm:cocptaction}
    If $\Gamma$ is any 2-dimensional extended presentation graph, then its edge groups embed, and $\sh_\Gamma$ acts cocompactly on a piecewise Euclidean $\cat(0)$ cell complex with cell stabilizers the conjugates of $\sh_\Lambda$ for spherical-type $\Lambda \leq \Gamma$.
\end{restatable}

This complex is the Shephard group analogue of the \emph{Deligne complex} for Artin groups and the \emph{Davis-Moussong complex} for Coxeter groups. We give the full definition and proof of Theorem \ref{thm:cocptaction} in Section \ref{sec:delignecplx}.
This allows one to adapt techniques used for 2-dimensional Artin groups to show similar properties of 2-dimensional Shephard groups. As an example, the next Theorem is based on the analogous result for 2-dimensional Artin groups \cite{vaskou2022acylindrical}.

\begin{restatable}{theorem}{acylindricallyhyperbolic} \label{thm:acylindricallyhyperbolic}
Suppose $\Gamma$ is an extended presentation graph satisfying
    \begin{enumerate}
        \item $|V(\Gamma)| \geq 3$,
        \item $\Gamma$ is 2-dimensional,
        \item $\sh_\Gamma$ does not split as a direct product (i.e., it is irreducible), and
        \item every connected component of $\Gamma$ has an edge $e$ such that $\sh_e$ is infinite.
    \end{enumerate}
    Then $\sh_\Gamma$ is acylindrically hyperbolic.
\end{restatable}

Past adapting proofs from 2-dimensional Artin groups, the condition that the vertex labels of $\Gamma$ are finite 
allows us to show novel results, that were either unknown or untrue for Artin groups. Our primary example is:

\begin{restatable}{theorem}{relativelyhyperbolic}\label{thm:relativelyhyperbolic}
    Suppose $\Gamma$ is a hyperbolic-type, 2-dimensional extended presentation graph, and let $\c P = \{\, \sh_\Lambda : |V(\Lambda)| = 2,\ W_\Lambda \text{ finite},\ \sh_\Lambda \text{ infinite} \,\}$, the collection of spherical-type edges of $\Gamma$ which give rise to infinite Shephard groups. Then $(\sh_\Gamma, \c P)$ is a relatively hyperbolic group pair. In particular, if every edge group $\sh_\Lambda$ is finite, then $\sh_\Gamma$ is hyperbolic.
\end{restatable}

The fact that the stabilizers of the $\cat(0)$ cell complex from Theorem \ref{thm:cocptaction} are precisely the edge (and vertex) subgroups indicate that in some sense, the poison subgroups are the only obstruction to being non-positively curved. The above theorem makes this idea more rigorous.

This Theorem is notable, because Artin groups are rarely hyperbolic relative to spherical-type subgroups; see \cite{kapovich2004relative} for a discussion on why this is. However, it was shown in \cite{charney2007relative} that Theorem \ref{thm:relativelyhyperbolic} holds for Artin groups if relative hyperbolicity is replaced with ``weak relative hyperbolicity'' (which we will not define here). Upgrading this to relative hyperbolicity has the following noteworthy consequences:
\begin{restatable}{corollary}{relhypconsequences} \label{cor:relhypconsequences}
    Suppose $\Gamma$ is a hyperbolic-type 2-dimensional extended presentation graph. Then $\sh_\Gamma$
    \begin{enumerate}
        \item has solvable word problem,
        \item satisfies the Tits alternative,
        \item has finite asymptotic dimension, and
        \item has the rapid decay property.
    \end{enumerate}
    If, in addition, there is no edge $\{i,j\}$ of $\Gamma$ with $1/p_i + 2/m_{ij} + 1/p_j = 1$, then $\sh_\Gamma$ is biautomatic.
\end{restatable}

If we restrict ourselves slightly further, it turns out we also have the following very desirable property for a large subclass of 2-dimensional Shephard groups.

\begin{restatable}{corollary}{residuallyfinite} \label{cor:residuallyfinite}
    Suppose $\Gamma$ is an extended presentation graph with no 3-cycle (or is ``triangle-free'') and no 4-cycle whose edges are each labeled $2$. Then $\sh_\Gamma$ is residually finite.
\end{restatable}

These are precisely the 2-dimensional graphs $\Gamma$ such that $\Gamma$ is hyperbolic-type and ``type FC'' (the spherical-type subgraphs are exactly the complete subgraphs).

One of the more broadly appealing aspects of Shephard groups with finite vertex labels are their potential to be used to prove results for Artin groups. For example,

\begin{restatable}{theorem}{artinresiduallyfinite} \label{thm:artinresiduallyfinite}
    Suppose $\Gamma$ is a triangle-free presentation graph with no 4-cycle with all edges labeled $2$. Then $A_\Gamma$ is residually finite.
\end{restatable}

Residual finiteness is unknown for most Artin groups, and was previously unknown even for most 2-dimensional Artin groups. For a discussion on residual finiteness of Artin groups, including some past results see \cite{kasia22}.

\subsection{Organization of paper}

In Section \ref{sec:background} we recall the relationship between central extensions and cohomology in preparation for Theorem \ref{thm:main2gen}. 
Then in Section \ref{sec:centext} we complete the proof of this Theorem by detailing the geometry of the dihedral Shephard groups through the perspective of central extensions. In Section \ref{sec:syllable}, we prove the main technical lemma which implies Theorem \ref{thm:cocptaction}, namely that the dihedral Shephard groups satisfy a ``syllable length'' condition similar to that enjoyed by dihedral Artin groups. Following this, we construct the complex from and complete the proof of Theorem \ref{thm:cocptaction} in Section \ref{sec:delignecplx}. The consequences of this complex being $\cat(0)$ are discussed in the remaining sections: Section \ref{sec:acylhyp} deals with acylindrical hyperbolicity, Section \ref{sec:relhyp} deals with relative hyperbolicity and its consequences, and Section \ref{sec:artinresid} deals with residual finiteness of the related Artin groups.

\subsection*{Acknowledgements}

The author would like to thank Mike Davis, Jingyin Huang, and Piotr Przytycki for their valuable discussions and input.
This work is partially supported by NSF grant DMS-2402105.

\section{Central Extensions}
\label{sec:background}

We begin by recalling some background on central extensions of groups and how they relate to cohomology.
Recall that if $A$ is abelian, an ($A$-)central extension of a group $G$ is a group $\tilde G$ fitting into a short exact sequence
\[\begin{tikzcd}[column sep=1.5em]
    0 \arrow[r] & A \arrow[r] & \tilde G \arrow[r] & G \arrow[r] & 0,
\end{tikzcd}\]
with the image of $A$ contained in the center of $\tilde G$. If $\tilde G_1$ and $\tilde G_2$ are $A$-central extensions of $G$, then we say they are ($A$-)\emph{equivalent} if there is an isomorphism $f : \tilde G_1 \to \tilde G_2$ making the diagram
\[\begin{tikzcd}[column sep=1.5em, row sep=0.5em]
     &  & \tilde G_1 \arrow[rd] \arrow[dd, "f"] \\
     0 \arrow[r] & A \arrow[rd]\arrow[ru] & & G \arrow[r] & 0 \\
     & & \tilde G_2 \arrow[ru]  
\end{tikzcd}\]
commute.

\begin{notat}
    For an arbitrary group $G$, we will denote its identity element by $e$. For an abelian group $A$, we will denote its identity element by $0$, and if it is cyclic, we will let $1$ denote a generator. 
\end{notat}

Let $E(G,A)$ denote the set of equivalence classes of $A$-central extensions of $G$ under $A$-equivalence. 
It is well known that there is a natural abelian group structure on $E(G,A)$, under which $E(G,A) \cong H^2(G;A)$ (for example, see \cite[Thm.~3.12]{brown1982cohomology} or \cite[\S 5.9.6]{dructu2018geometric}). 
There is a natural way to view the map $H^2(G;A) \to E(G,A)$.
Fix a presentation $G = \langle\, S \mid R \,\rangle$ and let $Y$ be a $K(G,1)$ such that the 2-skeleton $Y^2$ is the presentation complex for the given presentation of $G$. Let $C_i$ be the free abelian group on the $i$-cells of $Y$. Then in particular, the generators of $C_1$ are $S$ and the generators of $C_2$ are $R$. Let $d_i : C_i \to C_{i-1}$ be the standard cellular boundary map, and define $Z_i = \ker(d_i)$ and $B_i = \im(d_{i+1})$, so that $H_i(G) = Z_i / B_i$. 
We then dualize to obtain $C^i = \Hom(C_i, A)$, $d^i : C^i \to C^{i+1}$ defined by $\varphi \mapsto \varphi \circ d_{i+1}$, $Z^i = \ker(d^i)$, $B^i = \im(d^{i-1})$, and $H^n(G; A) = Z^i/B^i$.  

Now choose a class $[\varphi] \in H^2(G;A) = Z^2/ B^2$ and a representative $\varphi \in Z^2 \subseteq \Hom(C_2, A)$ (or in other words, some morphism $\varphi$ from the free abelian group on $R$ to $A$)%
.
Suppose $A = \langle\, T \mid Q \,\rangle$ is a presentation for $A$. 
Let $\tilde R = \{\,r^{-1}\varphi(r) : r \in R \,\}$ and $C = \{\, [s,t] : s \in S, t \in T \,\}$, subsets of the free group on the disjoint union $S \amalg T$. 
We then define a group 
\begin{equation}
    G_{\varphi} = 
    \langle\, S \amalg T \mid \tilde R \amalg Q \amalg C \,\rangle.
    \label{eq:gpextpres}
\end{equation}
It is straightforward to see that the map $[ \varphi ] \to G_{\varphi}$ is a well-defined injective homomorphism from $H^2(G;A)$ to $E(G,A)$. The construction of an inverse to this map is standard and unnecessary for our purposes, so we omit it, but it may be found in \cite[\S 5.9.6]{dructu2018geometric}.
We will call $[\varphi]$ the \emph{Euler class} of $G_\varphi$, and denote it by $e(G_\varphi) = [\varphi]$.

One standard result is
\begin{prop} \label{prop:trivialsplit}
    The underlying set of $G_\varphi$ is in bijection with $A \times G$, but $G_\varphi$ is isomorphic to $A \times G$ as a group if and only if $[\varphi] = 0$ in $H^2(G;A)$.
\end{prop}

We are interested in group-theoretic properties of central extensions induced by cohomological properties of their Euler classes. One first step in this direction is provided by this Lemma:

\begin{lemma} \label{lem:multoncoimpliesfinind}
    Suppose $\tilde G$ and $\tilde H$ are $A$-central extensions of a group $G$ such that $e(\tilde H) = n e(\tilde G)$, where $n$ is not a zero divisor of the \bb{Z}-module $A$. Then $\tilde H$ is isomorphic to a subgroup of $\tilde G$ with index $[A : nA]$. 
\end{lemma}
\begin{proof}
    Let $[\psi] = e(\tilde H)$ and $[\varphi] = e(\tilde G)$ with representatives $\psi$ and $\varphi$, resp., chosen such that $\psi = n \varphi$.
    We will write $G_\psi \cong \tilde H$ and $G_\varphi \cong \tilde G$, using the  notation given above for the presentation of $G_\psi$ and $G_\varphi$.

    Define a map $\Psi : G_\psi \to G_\varphi$ by $\Psi(s) = s$ for $s \in S$ and $\Psi(t) = nt$ for $t \in T$. 
    To check that this is a well-defined homomorphism we need to verify that $\Psi(r) = e$ for all $r \in \tilde R \amalg Q \amalg T$. 
    This is immediate for $r \in Q$ or $r \in C$, so we only need to verify this for $r \in \tilde R$; that is, we must verify that $\Psi(r) = \varphi(\Psi(r))$ for all $r \in R$. But this is also immediate since 
    \begin{align*}
        \Psi(r) 
        &= r \tag{$r$ is a word in $S$} \\
        &= \varphi(r) \\
        &= n\psi(r) \\
        &= \Psi(\psi(r)). \tag{$\psi(r)$ is a word in $T$}
    \end{align*}
    Thus this map is a well-defined homomorphism.
    Since $n$ is not a zero divisor, it is also injective. It is clear from the definition of $\Psi$ that its image has index $[A : nA]$ in $G_\varphi$.
\end{proof}

Now consider a group homomorphism $f: H \to G$. This induces in a standard way a morphism $f^* : H^2(G;A) \to H^2(H;A)$ on cohomology and thus a map $f^* : E(G,A) \to E(H,A)$ on central extensions via $G_\varphi \mapsto H_{f^*\varphi}$. In addition, $f$ can be lifted to a morphism $f_\varphi : H_{f^*\varphi} \to G_\varphi$. This gives the commutative diagram
\[
\begin{tikzcd}
    H_{f^*\varphi} \arrow[r, "f_\varphi"] \arrow[d] & G_\varphi \arrow[d] \\
    H \arrow[r, "f"'] & G
\end{tikzcd}
\]
As a consequence of the above diagram commuting, we have:

\begin{lemma}  \label{lem:imfinindimpliesliftfinind}
    For all $[\varphi] \in H^2(G;A)$, $[G_\varphi : f_\varphi(H_{f^*\varphi})] = [G : f(H)]$.
    In particular, if $f(H)$ has finite index in $G$, then $f_\varphi(H_{f^*\varphi})$ has finite index in $G_\varphi$.
\end{lemma}

\begin{proof}
Denote the rightmost vertical morphism of the diagram by $\Phi : G_\varphi \to G$. Then $f_\varphi(H_{f^*\varphi}) \cong \Phi^{-1}(f(H))$. Since $G_\varphi = \Phi^{-1}(G)$ and $\Phi$ is a surjective morphism, it follows that $[G_\varphi : f_\varphi(H_{f^*\varphi})] = [\Phi^{-1}(G) : \Phi^{-1}(f(H))] = [G : f(H)]$.
\end{proof}

It is a standard fact that if $f$ is injective, so is $f_\varphi$ \cite[Ex.~5.140.3]{dructu2018geometric}, so in particular, if $H$ is a finite index subgroup of $G$, then for any central extension $\tilde G$ of $G$, there is a finite index subgroup $\tilde H$ which is a central extension of $H$ over the same central copy of $A$.

We also know:
\begin{prop} \label{prop:orderofkernel} \textup{\cite[\S 3.G]{hatcher2002algebraic}}
    If $n = [G: f(H)]$, then $\ker(f^*)$ consists only of elements whose order divides $n$.
\end{prop}
In particular, $f^*$ is always injective on the free part of $H^2(G;A)$. 

This leads us to:

\begin{prop} \label{prop:infinitenotcat0}
    Suppose $G$ is finitely generated and $A \cong \bb[n]{Z}$ for some $n$. If $\tilde G$ is an $A$-central extension of $G$ such that $e(\tilde G)$ has infinite order, then $\tilde G$ cannot act properly by semi-simple isometries on a $\cat(0)$ space. In particular $\tilde G$ is not $\cat(0)$. 
\end{prop}

\begin{proof}
    We want to utilize \cite[Thm.~II.6.12]{bridson2013metric}, which states that if
    a finitely generated group $K$ acts by isometries on a $\cat(0)$ space, and if $K$ contains a central copy of $\bb[d]{Z}$ which acts faithfully by hyperbolic isometries (save for the identity), then this copy of $\bb[d]{Z}$ is virtually a direct factor of $K$. 

    To apply this to our situation, let $\tilde H$ be a finite index subgroup of $\tilde G$ containing %
    the canonical copy of $A \cong \bb[n]{Z}$ contained in $Z(\tilde G)$. 
    We claim that there is no subgroup $H'$ of $\tilde H$ such that $\tilde H \cong H' \times A$. 
    In order for such a subgroup to exist, we must have $A \subseteq Z(\tilde H)$. But then $H = \tilde H/A$ is a subgroup of $G$, and since $\tilde H$ has finite index in $\tilde G$, we know $H$ has finite index in $G$. In particular, if $\iota : H \to G$ denotes the inclusion morphism, then Proposition \ref{prop:orderofkernel} tells us that $e(\tilde H) = \iota^*e(\tilde G)$ is nontrivial in $H^2(H;A)$ since $e(\tilde G)$ has infinite order. Since this class is non-trivial, $\tilde H$ does not split as a direct product by Proposition \ref{prop:trivialsplit}. 

    By \cite[Thm.~II.6.12]{bridson2013metric}, we conclude that if $\tilde G$ acts by semi-simple isometries on a $\cat(0)$ space $X$, then the central copy of $A$ in $\tilde G$ cannot act faithfully by hyperbolic isometries.
    If the action is not faithful, then since $A \cong \bb[n]{Z}$ the action cannot be proper.
    So assume the action is not by hyperbolic isometries, meaning $A$ contains some non-trivial elliptic isometry $\gamma$ of $X$ (parabolic isometries are not semi-simple). Since $A \cong \bb[n]{Z}$, we know $\langle \gamma \rangle \cong \bb{Z}$. Since $\gamma$ is elliptic, then $\langle \gamma \rangle \cong \bb{Z}$ fixes a point of $X$, and hence does not act properly. 

    The fact that $\tilde G$ cannot be $\cat(0)$ follows from \cite[Prop.~II.6.10(2)]{bridson2013metric}. (Recall that a group is $\cat(0)$ if it acts properly and cocompactly by isometries on a $\cat(0)$ space.)
\end{proof}

\subsection{Central products}

For some later results, it becomes useful to use a commutative version of the amalgamated product.

\begin{defn}
    Let $G_1$, $G_2$, and $Z$ be any groups equipped with injective homomorphisms $\theta_i : Z \to Z(G_i)$. Let $N = \{\, (\theta_1(z), \theta_2(z)^{-1}) : z \in Z\,\} \leq G_1 \times G_2$. We define the \emph{amalgamated direct product} or \emph{central product} of $G_1$ and $G_2$ over $Z$ by
    \[
        G_1 \times_Z G_2 \coloneqq (G_1 \times G_2)/N
    \]
\end{defn}

Note that $N$ is a central subgroup of $G_1 \times G_2$, so this construction always gives a well-defined group.

Take subgroups $H_i$ of $G_i$ and let $N' = (H_1 \times H_2) \cap N$ and $Z' = \theta_1^{-1}(H_1) \cap \theta_2^{-1}(H_2) \subseteq Z$.
Note that
$N' = \{\, (\theta_1(z), \theta_2(z)^{-1}) : z \in Z' \,\}$.
By the isomorphism theorems,
\[
    H_1 \times_{Z'} H_2 = (H_1 \times H_2)/N' \cong (H_1 \times H_2)N/N \leq (G_1 \times G_2)/N. 
\] 
This demonstrates $H_1 \times_{Z'} H_2$ as a natural subgroup of $G_1 \times_Z G_2$. By the standard proof of the isomorphism theorems, the map that realizes this inclusion is $(h_1,h_2)N' \mapsto (h_1, h_2) N$. In addition, one readily sees that if $H_i$ is finite index in $G_i$ ($i = 1,2$), then $H_1 \times_{Z'} H_2$ is finite index in $G_1 \times_Z G_2$. 

As a special case, we can take $H_1 = G_1$ and $H_2 = \{\id_{G_2}\}$, in which case $Z' = \{\id_Z\}$ (or similarly, $H_1 = \{\id_{G_1}\}$ and $H_2 = G_2$). It is clear that $G_1 \cong G_1 \times_{Z'}  \{\id_{G_2}\}$ via the map $g \mapsto (g, \id_{G_2})$, since $Z'$ is trivial. We then apply our previous remarks to this situation to obtain

\begin{prop}
    \label{prop:factofdirectembed}
    For any groups $G_i$ and $Z$ as above, the groups $G_1$ and $G_2$ embed into $G_1 \times_Z G_2$ via the maps $\epsilon_1 : g \mapsto (g,\id_{G_2})N$ and $\epsilon_2 : g \mapsto (\id_{G_1},g)N$, respectively.
\end{prop}

Just as one has the notion of internal direct product, we may define an internal central product.

\begin{prop} \label{prop:internalcentralproduct}
    Suppose $G$ is any group with subgroups $H_1, H_2$ such that $h_1h_2 = h_2h_1$ for all $h_i \in H_i$ ($i = 1,2$). Let $Z = H_1 \cap H_2$. Then $H_1H_2 \leq G$ is isomorphic to $H_1 \times_Z H_2$.
\end{prop}

\begin{proof}
    First note that since the elements of $H_1$ and the elements of $H_2$ commute, $Z \subseteq Z(H_i)$ for $i = 1,2$. Let $\iota_i : Z \to H_i$ ($i = 1,2$) denote the inclusion map. Let $\alpha : H_1 \times H_2 \to G$ be given by $(h_1,h_2) \mapsto h_1h_2$. Clearly $\alpha$ is a surjective homomorphism. The kernel of $\alpha$ is $\{\,(h_1,h_2) : h_1 = h_2^{-1}\,\}$, or rewritten, 
    \[
        \{\,(h_1,h_2^{-1}) : h_1 = h_2, h_i \in H_i\,\}
        = \{\,(\iota_1(h),\iota_2(h)^{-1}) : h \in H_1 \cap H_2 = Z\,\}.
    \]
    Thus
    \[
        H_1H_2 \cong (H_1 \times H_2)/\ker(\alpha) \cong H_1 \times_Z H_2. \qedhere
    \]
\end{proof}

We need one final property of central products before we can continue.

\begin{prop} \label{prop:usefulcounting}
    Take $G_i$ and $Z$ as above, and identify $G_i$ with the subgroup of $G_1 \times_Z G_2$ as in Proposition \ref{prop:internalcentralproduct}. Then $[G_1 \times_Z G_2 : G_1] = [G_2 : Z]$ (and similarly, $[G_1 \times_Z G_2 : G_2] = [G_1 : Z]$).
\end{prop}

This follows immediately from Proposition \ref{prop:internalcentralproduct} and standard counting results ($[G_1G_2 : G_1] = [G_2 : G_1 \cap G_2] = [G_2 : Z]$).

\section{Dihedral Shephard groups}
\label{sec:centext}

The goal of this section is to study the geometry of the infinite dihedral Shephard groups in detail. We display them as \bb{Z}-central extensions of infinite triangle groups whose Euler class has infinite order, proving the first statement of Theorem \ref{thm:main2gen}. Following this, we determine finer properties which we believe are interesting in their own right (and in particular complete the proof of Theorem \ref{thm:main2gen}).

\subsection{Dihedral Shephard groups as central extensions}
\label{subsec:dihedralcentral}

We start by examining the second integral cohomology of the triangle groups.
Let $(p,q,r)$ be a triple of positive integers, and define $h = \frac{1}{p} + \frac{1}{q} + \frac{1}{r}$. 
We let $\bb{Y} = \bb{S}, \bb{E}$, or $\bb{H}$ when $h > 1$, $h = 1$, or $h < 1$, respectively. 
Define a triangle $T(p,q,r)$ in $\bb[2]{Y}$ with vertices $a$, $b$, and $c$, and angles $\pi/p$, $\pi/q$, $\pi/r$ at $a$, $b$, and $c$, respectively.
The triangle group $\Delta(p,q,r)$ is the group generated by rotations of angle $2\pi/p$, $2\pi/q$, and $2\pi/r$ about the vertices $a$, $b$, and $c$, resp., of $T$ in \bb[2]{Y}. 
(We note that sometimes $\Delta(p,q,r)$ is called a \emph{von Dyck group} and ``triangle group'' is sometimes used to refer to 3-generator Coxeter groups.)
This group has the well-known presentations
\begin{align*}
    \Delta(p,q,r) &= \langle\, a,b,c \mid a^p = b^q = c^r = abc = e \,\rangle \\
    &= \langle\, a,c \mid a^p = (ac)^q = c^r = e \,\rangle
\end{align*}
where, by abusing notation, the generators $a$, $b$, and $c$ correspond to the respective aforementioned rotations about the vertices $a$, $b$, and $c$ of $T$. 
We will use the second presentation to compute the cohomology.
Since we're interested in only the second integral cohomology, we describe a construction of the 3-skeleton of a $K(\Delta,1)$.

First let $K^{(2)}$ be the presentation complex for $\Delta(p,q,r)$; that is, the cell complex with one 0-cell $x$, two (oriented) 1-cells labeled $a$ and $b$, and three 2-cells labeled $e_a$, $e_c$, and $e_{ac}$. 
The attaching map of the cell $e_a$ takes the boundary $\partial e_a$ to the loop $a^p$ with positive orientation, and similarly $e_c$ and $e_{ac}$ are attached to the loops $c^r$ and $(ac)^q$, respectively.

Let $\widetilde K^{(2)}$ denote the universal cover of $K^{(2)}$. Note that $\widetilde K^{(2)}$ is the Cayley 2-complex of $\Delta(p,q,r)$ and the 1-skeleton of $\widetilde K^{(2)}$ is the Cayley graph of $\Delta(p,q,r)$. This Cayley graph is the 1-skeleton of the semiregular tiling $\c T = \c T(p,2q,r)$ of \bb[2]{Y} by $p$-gons, $2q$-gons, and $r$-gons (e.g., \cite{MorenoStypa}). See Figure \ref{fig:334cayley} for an example. The cell structure on $\widetilde K^{(2)}$ is obtained from this tiling $\c T$ by gluing $(n-1)$ extra cells to the boundary of each $n$-gon. 

\begin{figure}
    \centering
    \includegraphics[width=0.45\textwidth]{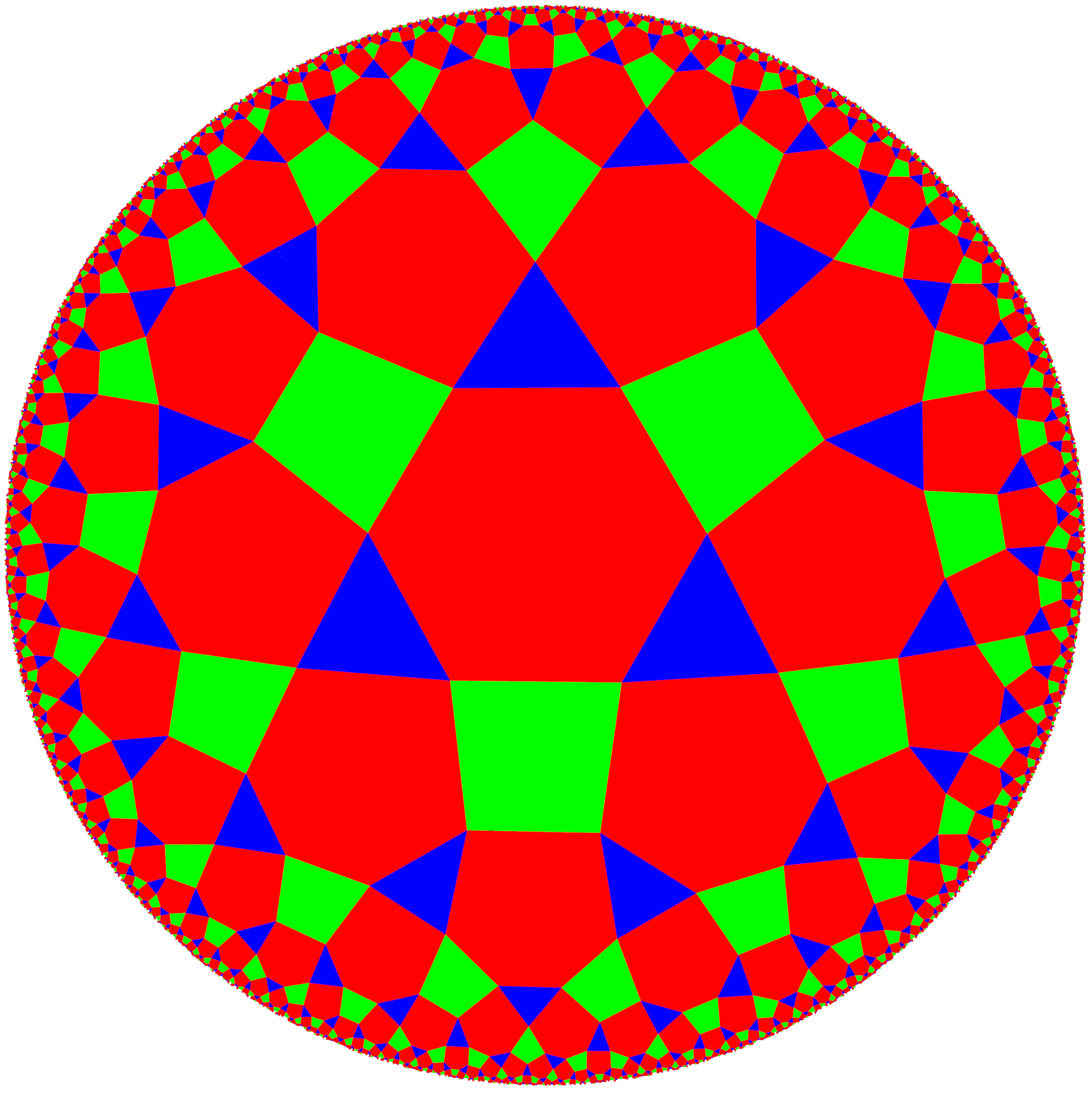}
    \caption{The tiling $\c T(3,6,4)$ whose 1-skeleton is the Cayley graph of $\Delta(3,3,4)$ \cite{Kaleido}}
    \label{fig:334cayley}
\end{figure}

We now obtain the cell complex $\widetilde K$ by filling the generators of $\pi_2(\widetilde K^{(2)})$ as follows. Choose an $n$-gon $E$ of $\c T$, and label all cells attached to $\partial E$ in $\widetilde K^{(2)}$ (including $E$) by $E_1,\dots, E_n$. 
The cells $E_i \cup E_{i+1}$ (with indices mod $n$) form a sphere in $\widetilde K^{(2)}$, so attach a 3-cell $E_{i, i+1}$ to this sphere. 
We will denote $E_S = \bigcup E_{i, i+1}$. Notice that $E_S$ is homeomorphic to a 3-sphere. 
We obtain $\widetilde K$ from $\c T$ by replacing each $n$-gon $E$ with the sphere $E_S$. 
Notice that $\widetilde K$ is the wedge of the spheres $E_S$. In particular, $\pi_2(\widetilde K) = 0$.

There is a natural action of $\Delta(p,q,r)$ on $\widetilde K$ coming from the action on $\widetilde K^{(2)}$ by deck transformations. For an $n$-gon of $\c T$, The stabilizer of the set $\bigcup E_i$ is conjugate to exactly one of $\langle a \rangle$, $\langle ac \rangle$, or $\langle c \rangle$ (depending on if $n = p$, $2q$, or $r$, respectively). 
The action of this stabilizer on $E_S$ is simply the standard action of $\bb{Z}/n\bb{Z}$ on \bb[3]{S}. This action is still free and properly discontinuous, so we may define the quotient space $K = \widetilde K / \Delta$ such that $\widetilde K$ is the universal cover of $K$. Note that $K^{(2)}$ (as defined before) is the 2-skeleton of $K$.

Since $\widetilde K$ is the universal cover of $K$, we know $\pi_2(K) \cong \pi_2(\widetilde K) = 0$.
This means that $H_2(\Delta) \cong H_2(K)$ and $H^2(\Delta; \bb{Z}) \cong H^2(K; \bb{Z})$, where $\Delta = \Delta(p,q,r)$.
Note that for any $n$-gon $E$ of $\c T$, the cell $E_i$ of $\widetilde K$ maps to exactly one of the cells $e_a$, $e_{ac}$, or $e_c$ if $n = p$, $2q$, or $r$, respectively. 
In particular, for all such $E$, the 3-cell $E_S$ in $\widetilde K$ maps to a single 3-cell of $K$ which we will denote $f_a$, $f_{ac}$, and $f_c$ if $n = p$, $2q$, or $r$, resp. Topologically, $\partial f_g = e_g \cup g \cup x$ if $g = a$ or $c$, and $\partial f_{ac} = e_{ac} \cup a \cup c \cup x$.
Moreover, the closure $\overline f_g$ of $f_g$ is the $3$-dimensional lens space of order $|g|$ with its standard cell structure for $g = a,b$ and $\overline f_{ac}$ is the $3$-dimensional lens space of order $q$ with two points identified at $x$. 

To fix notation for the computation of the cohomology of $K$, let 
$C_n$ denote the free abelian group on the $n$-cells of $K$ and $d_n : C_n \to C_{n-1}$ the standard cellular boundary map. We then let $Z_n = \ker(d_{n})$ and $B_n = \im(d_{n+1})$ so that $H_n(K) = Z_n / B_n$. As usual, we dualize to obtain
$C^n = \Hom(C_n, \bb{Z})$, $d^n = (d_{n+1})_*: C^{n} \to C^{n+1}$ given by $\varphi \mapsto \varphi \circ d_{n+1}$, $Z^n = \ker(d^n)$, $B^n = \im(d^{n-1})$, so that $H^n(K;\bb{Z}) = Z^n/B^n$.

By a computation identical to that of the lens spaces (e.g., \cite[Ex.~2.43]{hatcher2002algebraic}), we see that $d_3$ is the zero map and
\begin{align*}
    d_2(e_a) &= pa \\
    d_2(e_c) &= rc \\
    d_2(e_{ac}) &= q(a+c) \\
\end{align*}
This allows one to easily compute that $B_2 = 0$ and $Z_2$ (hence $H_2(K)$) is infinite cyclic generated by 
\[\frac{\lcm(p,q,r)}{q}e_{ac} - \frac{\lcm(p,r)}{p}e_a - \frac{\lcm(p,r)}{r}e_c.\]

\begin{lemma} \label{lem:shepinforder}
    Define a map $\varphi \in C^2$ by 
\begin{align*}
    \varphi(e_a) &= 0 \\
    \varphi(e_c) &= 0 \\
    \varphi(e_{ac}) &= 1. 
\end{align*}
    Then the class $[\varphi] \in H^2(K,\bb{Z})$ has infinite order.
\end{lemma}
\begin{proof}
Let $\rho : Z_2 \to Z_2/B_2$ denote the quotient map and $R : C^2 \to \Hom(Z^2; \bb{Z})$ denote the restriction map $\psi \mapsto \psi|_{Z_2}$. 

The restriction $\varphi|_{Z_2}$ is nontrivial, and $B_2 = 0$, so $\rho \circ R\circ \varphi$ gives a nontrivial element of $\Hom(H_2; \bb{Z})$.  
By the universal coefficient theorem, the free part of $H^2(K;\bb{Z})$ is isomorphic to $\Hom(H_2;\bb{Z})$ via the map $\psi \mapsto \rho \circ R \circ \overline \psi$ (where $\overline \psi$ is some lift of $\psi$ along the quotient map $Z^2 \to Z^2/B^2$), so it follows that the image of $\varphi$ in $H^2(K;\bb{Z})$ has infinite order.
\end{proof}

\begin{lemma} \label{lem:shepcentral}
    Let $\varphi$ be as above. Then $\sh(p,2q,r) \cong \Delta_\varphi$. In particular, $\sh(p,2q,r)$ has infinite cyclic center generated by $(st)^q$.
\end{lemma}
\begin{proof}
By the definition of $\Delta_\varphi$, we see that it has presentation
\[
    \langle\, a,c,z \mid a^p = c^r = e, (ac)^q = z, [z,a] = [z,c] = e \,\rangle.
\]
Since $z$ commutes with $a$ and $b$,
\[
    z = a^{-1}za = a^{-1}(ac)^qa = (ca)^q,
\]
so we just as well may write
\[
    \langle\, a,c,z \mid a^p = c^r = e, (ac)^q = (ca)^q = z, [z,a] = [z,c] = e \,\rangle.
\]
But then notice that 
\begin{align*}
    &(ac)^qa = a(ca)^q = a(ac)^q, \qquad \text{ and} \\
    &c(ac)^q = (ca)^qc = (cb)^qc,
\end{align*}
so the relations $[z,a] = [z,c] = 1$ are now redundant, and we may write
\[
    \langle\, a,c,z \mid a^p = c^r = e, (ac)^q = (ca)^q = z \,\rangle.
\]
But now the generator $z$ is redundant, so we arrive at 
\[
    \Delta_\varphi = \langle\, a,c\mid a^p = c^r = e, (ac)^q = (ca)^q\,\rangle \cong \sh(p,2q,r).
\]
Since $(ac)^q = z \in Z(\Delta_\varphi)$ and $\Delta_\varphi/\langle z \rangle \cong \Delta$ is centerless (because $h \leq 1$), it follows that $(ac)^q = (st)^q$ generates the center of $\sh(p,2q,r)$.
\end{proof}

Next, we examine the Shephard groups $\sh(p,q,p)$ for $q$ odd. 

\begin{lemma}
    The subgroup of $\sh(p,2q,2)$ generated by $s$ and $tst$ is index-2 and isomorphic to $\sh(p,q,p)$.
\end{lemma}
\begin{proof}
First, we fix the presentations of these groups as
\begin{align*}
    \sh(p,q,p) &= \langle\,  \sigma,\tau \mid \sigma^p = \tau^p = e, \, \underbrace{\sigma\tau\sigma \ldots}_{q \text{ letters}} = \underbrace{\tau\sigma\tau\ldots}_{q \text{ letters}} \,\rangle \\
    \sh(p,2q,2) &= \langle\, s,t \mid s^p = t^2 = e, \, (st)^q = (ts)^q \,\rangle.
\end{align*}
Let the generator $\upsilon$ of $\bb{Z}/2\bb{Z}$ act on $\sh(p,q,p)$ by interchanging $\sigma$ and $\tau$. Then the semidirect product of $\sh(p,q,p)$ and $\bb{Z}/2\bb{Z}$ under this action has the presentation
\[
    \sh(p,q,p) \rtimes \bb{Z}/2\bb{Z} = 
    \langle\, \sigma, \tau, \upsilon \mid \sigma^p = \tau^p = e, \underbrace{\sigma\tau\sigma \ldots}_{q \text{ letters}} = \underbrace{\tau\sigma\tau\ldots}_{q \text{ letters}}\,,\ \sigma = \upsilon \tau \upsilon, \upsilon^2 = e \,\rangle.
\]
Since this is a semidirect product, the subgroup generated by $\sigma$ and $\tau$ is isomorphic to $\sh(p,q,p)$. We claim this product is isomorphic to $\sh(p,2q,2)$.

The relation $\sigma = \upsilon \tau \upsilon$ makes $\tau^p = e$ redundant, so we may remove it. In addition, notice that 
\begin{align*}
    \underbrace{\tau\sigma\tau \ldots}_{q \text{ letters}}
    &= (\tau\sigma)^{(q-1)/2}\tau \\
    &= (\upsilon\sigma\upsilon\sigma)^{(q-1)/2}\upsilon \sigma \upsilon \\
    &= (\upsilon\sigma)^{(q-1)}\upsilon \sigma \upsilon  \\
    &= (\upsilon\sigma)^{q} \upsilon.
\end{align*}
Similarly,
\begin{align*}
    \underbrace{\sigma\tau\sigma \ldots}_{q \text{ letters}} 
    &= (\sigma\tau)^{(q-1)/2}\sigma \\
    &= (\sigma\upsilon\sigma\upsilon)^{(q-1)/2}\sigma \\
    &= (\sigma\upsilon)^{(q-1)}\sigma.
\end{align*}
Thus the relation $\underbrace{\sigma\tau\sigma \ldots}_{q \text{ letters}} = \underbrace{\tau\sigma\tau\ldots}_{q \text{ letters}}$ is equivalent to
\[
    (\upsilon\sigma)^{q} \upsilon = (\sigma\upsilon)^{(q-1)}\sigma,
\]
which in turn is equivalent to
\[
    (\upsilon\sigma)^{q} = (\sigma\upsilon)^{(q-1)}\sigma\upsilon = (\sigma\upsilon)^q.
\]
Therefore, 
\[
    \sh(p,q,p) \rtimes \bb{Z}/2\bb{Z} = 
    \langle\, \sigma, \tau, \upsilon \mid \sigma^p = \upsilon^2 = e, (\upsilon\sigma)^{q} = (\sigma\upsilon)^q, \sigma = \upsilon \tau \upsilon \,\rangle.
\]
And with this presentation, $\tau$ is obviously redundant, so we see that 
\[
    \sh(p,q,p) \rtimes \bb{Z}/2\bb{Z} \cong \sh(p,2q,2)
\]
via the map $\sigma \mapsto s$ and $\upsilon \mapsto t$. 
\end{proof}

Since $\sh(p,q,p)$ viewed inside $\sh(p,2q,2)$ is generated by $\sigma = s$ and $\tau = tst$, we have that the element $(\sigma\tau)^q = (st)^{2q}$ is both in $Z(\sh(p,2q,2))$ and $Z(\sh(p,q,p))$, since $(st)^{2q} = (stst)^q = (\sigma\tau)^q$. 
The image of $\sh(p,q,p)$ under the quotient map $\sh(p,2q,2) \to \Delta(p,q,2)$ is the subgroup $D$ of $\Delta(p,q,2)$ generated by $a$ and $cac$. This subgroup is finite index in $\Delta(p,q,2)$ (e.g., since $\sh(p,q,p)$ is finite index in $\sh(p,2q,2)$). In particular, 

\begin{lemma} \label{lem:oddinfinite}
    $\sh(p,q,p)$ is a \bb{Z}-central extension of $D$ via the image of the class $[\varphi]$ from $H^2(\Delta;\bb{Z})$ to $H^2(D,\bb{Z})$ (which has infinite order).
\end{lemma}

Note that this restriction of $[\varphi]$ has infinite order by Proposition \ref{prop:orderofkernel}.

As a brief aside, using this central extension structure, we can show

\begin{cor} \label{cor:dihedralfiniteconjugate}
    Suppose $\sh(p,q,r)$ is infinite and $g \in \sh(p,q,r)$ has finite order. Then $g$ is conjugate to a power of one of the standard generators of $\sh(p,q,r)$.
\end{cor}

\begin{proof}
    Let $s$ and $t$ denote the generators of order $p$ and $r$, resp., of $\sh(p,q,r)$. The central quotient $\sh(p,q,r)/Z$ of $\sh(p,q,r)$ acts properly and cocompactly on either \bb[2]{E} or \bb[2]{R}, where the generators $s$ and $t$ are sent to rotations $a$ and $c$, resp., by an angle of $2\pi/p$ and $2\pi/r$, resp. Since $g$ has finite order, it cannot be contained in the center of $\sh(p,q,r)$, and thus has non-trivial image $\overline g$ in $\sh(p,q,r)/Z$, still with finite order. 
    Then by standard facts about isometries of \bb[2]{E} and \bb[2]{H} (and in particular, facts about the triangle groups), $\overline g$ is conjugate (within $\sh(p,q,r)/Z$) to a power of either $a$ or $c$. Without loss of generality, we may assume $\overline g = \overline h^{-1} a^k \overline h$ for some $\overline h \in \sh(p,q,r)/Z$ and $0 < k < p$ (the argument for conjugates of powers of $c$ is identical). This means there is some lift $h \in \sh(p,q,r)$ of $\overline h$ such that $gZ = (h^{-1}s^k h)Z$. In particular, there is some $z \in Z$ such that $g = (h^{-1}s^k h)z = h^{-1}(zs^k)h$. If $z \not= e$, then $z$ has infinite order, thus so does $zs^k$ (since $z$ commutes with $s^k$), contradicting the assumption that $g$ has finite order. Thus $z = e$ and $g = h^{-1}s^kh$.
\end{proof}

We now prove the first part of Theorem \ref{thm:main2gen}, namely

\begin{thm}
    Let $(p,q,r)$ be a triple of integers each $\geq 2$, with $p = r$ if $q$ is odd. If $1/p + 2/q + 1/r \leq 1$, then $\sh(p,q,r)$ cannot admit a proper action by semi-simple isometries on any $\cat(0)$ space. In particular, $\sh(p,q,r)$ is not $\cat(0)$. 
\end{thm}

\begin{proof}
    First suppose $q$ is even, say $q = 2k$ for a positive integer $k$. 
    Then 
    \[
        \frac{1}{p} + \frac{1}{k} + \frac{1}{r} = \frac{1}{p} + \frac{2}{2k} + \frac{1}{r} = \frac{1}{p} + \frac{2}{q} + \frac{1}{r} \leq 1.
    \]
    So by Lemma \ref{lem:shepcentral}, $\sh(p,q,r) = \sh(p,2k,r)$ is a \bb{Z}-central extension of the infinite group $\Delta(p,k,r)$ via a second cohomology class (of infinite order by Lemma \ref{lem:shepinforder}) of $\Delta(p,k,r)$. Thus by Proposition \ref{prop:infinitenotcat0}, $\sh(p,q,r)$ cannot act properly by semi-simple isometries on a $\cat(0)$ space and is not $\cat(0)$.

    Now assume $q$ is odd (implying $p = r$). This means
    \[
        1 \geq \frac{1}{p} + \frac{2}{q} + \frac{1}{r} = \frac{2}{p} + \frac{2}{q},
    \]
    implying
    \[
        \frac{1}{p} + \frac{1}{q} \leq \frac{1}{2},
    \]
    and so
    \[
        \frac{1}{p} + \frac{1}{q} + \frac{1}{2} \leq 1.
    \]
    Thus $\sh(p,q,p)$ is a \bb{Z}-central extension of a finite index subgroup of $\Delta(p,q,2)$ via a second cohomology class of infinite order (Lemma \ref{lem:oddinfinite}). By Proposition \ref{prop:infinitenotcat0} cannot act properly by semi-simple isometries on a $\cat(0)$ space and is not $\cat(0)$. 
\end{proof}

\subsection{Further geometry of the extension}
\label{sec:furthergeom}

We can give insight into the geometry of the dihedral Shephard groups beyond the general fact of Proposition \ref{prop:infinitenotcat0}. Namely, we will discuss the remainder of Theorem \ref{thm:main2gen}.

To encompass both types of dihedral Shephard groups dealt with above (depending on the parity of $q$), we fix notation for this section.
First, we let $Sh = \sh(p,2q,r)$ or $Sh = \sh(p,q,p)$ (in which case we say $r = 2$). In the first case, we define $D = \Delta = \Delta(p,q,r)$, and in the second case, we define $D$ to be the subgroup of $\Delta = \Delta(p,q,2)$ generated by $a$ and $cac$. 
To summarize the results of the previous section in this notation, $Sh$ is a \bb{Z}-central extension of $D$ whose Euler class has infinite order, where $D$ acts geometrically on \bb[2]{E} or \bb[2]{H} if $h = 1$ or $h < 1$, resp., as a finite index subgroup of a triangle group $\Delta$. 

As a consequence of the latter point, $D$ contains a finite index torsion-free subgroup $M$ \cite[Thm.~2.7]{milnor19753}, and in particular $M$ must be a (closed) surface group.
By Lemma \ref{lem:imfinindimpliesliftfinind}, $M$ lifts to a finite index subgroup $M_\varphi$ of $Sh$. 
Since $M$ is a surface group, $H^2(M;\bb{Z}) \cong \bb{Z}$. Let $\widetilde M$ denote a central extension of $M$ such that $e(\widetilde M)$ is a generator of $H^2(M;\bb{Z})$. (Sometimes $\widetilde M$ is called the ``universal central extension'' of $M$, although this conflicts with the standard definition of universal central extension which applies only to perfect groups.)
When $h = 1$, $M \cong \bb[2]{Z}$ and $\widetilde M \cong H(3)$, the 3-dimensional integer Heisenberg group.
When $h < 1$, then $M$ is a hyperbolic surface group, and $\widetilde M$ is a uniform lattice in $\widetilde{\SL_2\bb{R}}$.
Since $e(\widetilde M)$ is a generator of the second cohomology, $e(M_\varphi)$ is a non-zero multiple of $e(\widetilde M)$.
By Lemma \ref{lem:multoncoimpliesfinind}, this means $M_\varphi$ is finite index in $\widetilde M$. Thus $Sh$ is commensurable to $\widetilde M$. As an immediate consequence, we have

\begin{prop} \label{prop:edgesresidfinite}
    For any triple $(p,q,r)$ of integers $\geq 2$ (with $p = r$ when $q$ is odd), the group $\sh(p,q,r)$ is linear.
\end{prop}
\begin{proof}
    When $\sh(p,q,r)$ is finite, this was shown in \cite{1975regular}. So, suppose $\sh(p,q,r)$ is not finite, i.e., $h = 1/p + 2/q + 1/r \leq 1$.
    It is an easy exercise to see that if $H$ is a finite index subgroup of $G$, then $G$ is linear if and only if $H$ is linear. 
    Thus if two groups $G$ and $H$ are commensurable, one is linear if and only if the other is. So it suffices to note that $\widetilde M$ is always linear:
    if $h = 1$, then $\widetilde M$ is the 3-dimensional integer Heisenberg group $H(3)$ (well-known to be linear), and, if $h < 1$, then $\widetilde M$ is linear by \cite[\S IV.48]{de2000topics} (via an explicit injection from $\widetilde M$ to $\SL_2 \bb{R} \times H(3)$).
\end{proof}

We will note here that the Shephard group analogue of the Tits representation used to show finite Shephard groups are linear in \cite{1975regular} is not faithful for infinite dihedral Shephard groups. (A quick computation shows that the center of the image under this representation is always finite.) Finding an explicit representation for the Shephard groups is straightforward using the information given above, so we leave it as an exercise. The $h < 1$ Shephard groups have an interesting explicit (but non-linear) representation as isometries of $\widetilde{\SL_2\bb{R}}$, which will discuss soon. 
First we examine the ``Euclidean-like'' case of $h = 1$.

\begin{prop}
    Suppose $h = 1$. Then $Sh$ is virtually nilpotent and is not semihyperbolic.
\end{prop}

\begin{proof}
    Virtual nilpotency of $Sh$ follows from the Q.I. rigidity of virtual nilpotency and the nilpotency of the 3-dimensional Heisenberg group.
    Moreover, the 3-dimensional Heisenberg group has cubic Dehn function, and thus so does $Sh$. Since semihyperbolic groups have at-most quadratic Dehn function, it follows that $Sh$ cannot be semihyperbolic.
\end{proof} 

This implies, for example, that if such a Shephard group embeds in an arbitrary Shephard group $\sh_\Gamma$, then $\sh_\Gamma$ is not semihyperbolic.

The case $h < 1$ is quite rich. For example, since the central quotient is word hyperbolic in this case, 
by \cite{neumann1997central}, this immediately implies

\begin{prop} \label{prop:negbiauto}
    If $h < 1$, then $Sh$ is biautomatic.
\end{prop}

Past knowing that such a Shephard group is commensurable to a uniform lattice in $\widetilde{\SL_2\bb{R}}$, we can also explicitly demonstrate it as a group of isometries of $\widetilde{\SL_2\bb{R}}$. In some sense this is ``more natural'' than the linearity of Proposition \ref{prop:edgesresidfinite}, because it directly generalizes the method in which the presentation for the finite dihedral Shephard groups are derived in \cite[\S 9]{1975regular}. (The main technical difference is the fact that $1/\lcm(p,q,r) \not= 1/p + 1/q + 1/r - 1$ when $1/p + 1/q + 1/r \leq 1$, unlike in the finite case.) 

\begin{prop} \label{prop:unilattice}
    If $h < 1$, then $Sh$ is a uniform lattice in $\isom(\widetilde{\SL_2\bb{R}})$.
\end{prop}

\begin{proof}
    Since $h < 1$ and each of $p$, $q$, and $r$ are finite, we know that $\Delta$, $D$, and $M$ are uniform lattices in $\scr G \coloneqq \PSL_2\bb{R}$. Recall that $\scr G$ is isometric to the unit tangent bundle of \bb[2]{H} (under the Sasaki metric), which itself can be thought of as a $U(1)$-bundle over \bb[2]{H}. This bundle is topologically trivial (since $\bb[2]{H}$ is contractible) but is well known to be metrically non-trivial. 
    Let $\widetilde{\scr{G}}$ ($= \widetilde{\SL_2\bb{R}}$) denote the universal cover of $\PSL_2\bb{R}$. Then $\widetilde{\scr{G}}$ is a (metrically non-trivial) \bb{R}-bundle over \bb[2]{H}.
    Note that $\widetilde M$ (as defined above) is actually the preimage of $M$ under the covering map $\widetilde{\scr G} \to \scr G$. Let $\widetilde \Delta$ denote the preimage of $\Delta$ under this covering map. 
    By \cite[Lem.~3.1]{milnor19753} this group has the presentation
    \[
        \widetilde \Delta = \langle\, \tilde a, \tilde b, \tilde c \mid \tilde a^p = \tilde b^q = \tilde c^r = \tilde a \tilde b \tilde c \,\rangle,
    \]
    where each of $\tilde a$, $\tilde b$, and $\tilde c$ are lifts of the respective rotations $a$, $b$, and $c$ to $\widetilde{\scr G}$.
     Moreover, the proof of said Lemma shows that $\tilde a \tilde b \tilde c$ generates the center of $\widetilde{\scr G}$ and the center of $\widetilde \Delta$. Note that $\widetilde \Delta$ is also a uniform lattice in $\widetilde{\scr G}$ since $\Delta$ is a uniform lattice in $\scr G$. 
    In order to display the dihedral Shephard groups as subgroups of $\isom(\widetilde{\scr G})$, we introduce another class of isometries.

    For $\theta \in \bb{R}$, define a map $r_\theta$ which acts on $\scr G$ by preserving the $U(1)$-fiber structure over $\bb[2]{H}$, such that $r_\theta$ projects down to the identity map of $\bb[2]{H}$ and rotates each fiber by $2\pi \theta$. 
    Since the bundle is topologically trivial, there is no issue with the existence and well-defined-ness of this map. 
    Moreover, it is clear that this map is an isometry for any $\theta$. We can also see that each $r_\theta$ commutes with the action of $\scr G$. 
    The group of all $r_\theta$ is isomorphic to $U(1)$. Each $r_\theta$ can be lifted to a map $\tilde r_\theta$ of $\widetilde{\scr G}$ which translates along the \bb{R}-fibers a common distance $2\pi\theta$. Note that this action commutes with the left action of $\widetilde{\scr G}$. 
    Since $r_\theta$ is an isometry of $G$, it follows that $\tilde r_\theta$ is an isometry of $\widetilde{\scr G}$. Let $\scr R = \{\, \tilde r_\theta : \theta \in \bb{R} \,\}$. As a straightforward exercise, one may verify that $\scr R$ along with the left-multiplication maps of $\widetilde{\scr G}$ generate the entirety of $\isom(\widetilde{\scr G})$.
    Since the elements of $\scr R$ commute with the elements of $\widetilde{\scr G}$ (and vice versa), this means $\isom(\widetilde{\scr G}) = \scr R\widetilde{\scr G}$. But note that $\widetilde{\scr G} \cap \scr R = \{\,\tilde r_\theta : \theta \in \bb{Z}\,\} \cong \bb{Z}$; so, by Proposition \ref{prop:internalcentralproduct}, $\isom(\widetilde{\scr G}) \cong \widetilde{\scr G} \times_{\bb{Z}} \scr R$.

    Now let $k = \mathrm{lcm}(p,q,r)$.
    Define $\widetilde \Delta_k$ to be the subgroup of $\isom(\widetilde{\scr G})$ generated by $\tilde a$, $\tilde b$, $\tilde c$, and $z \coloneqq \tilde r_{1/k}$. Since $\tilde r_{1/k}$ commutes with the left multiplication action of $\widetilde{\scr G}$, this group has the presentation
    \[
        \widetilde \Delta_k = \langle\, \tilde a, \tilde b, \tilde c, z \mid \tilde a^p = \tilde b^q = \tilde c^r = \tilde a \tilde b \tilde c = z^k, [\tilde a, z] =  [\tilde b, z] =  [\tilde c, z] = e \,\rangle.
    \]
    Note that this is isomorphic to a central product $\widetilde \Delta \times_{\bb{Z}} \langle \tilde r_{1/k} \rangle$.
    Since $\widetilde \Delta$ is a uniform lattice in $\widetilde{\scr G}$ and $\langle \tilde r_{1/k}\rangle$ is a uniform lattice in $\scr R$, it follows easily that $\widetilde \Delta_k$ is a uniform lattice in $\widetilde{\scr G} \times_{\bb{Z}} \scr R \cong \isom(\widetilde{\scr G})$.
    We claim that $\sh(p,2q,r)$ is isomorphic to a finite index subgroup of $\widetilde \Delta_k$, hence is a uniform lattice in
     $\isom(\widetilde{\scr G})$ as well. Since $\sh(p,q,p)$ is finite index in $\sh(p,2q,2)$, the result follows for these groups as well. 

    Let $m = \frac{k}{p} + \frac{k}{q} + \frac{k}{r} - k$, and consider the group
    \begin{equation}
        G = \langle\, s,t,\phi \mid s^p = t^r = e, (st)^{q} = (ts)^{q} = \phi^{-qm}, [s,\phi] = [t,\phi] = e \,\rangle. \label{2gen:eqn:defDelta}
    \end{equation} 
    Since $1/p + 1/q + 1/r = h \not= 1$, we know $m \not= 0$.
    Then $G$ is an amalgamated direct product of $\sh(p,2q,r)$ and $\langle \phi \rangle \cong \bb{Z}$ along the subgroup $\bb{Z} \cong \langle (st)^q \rangle \cong \langle \phi^{-qm} \rangle$. By Proposition \ref{prop:factofdirectembed},
    the subgroup of $G$ generated by $s$ and $t$ is isomorphic to $\sh(p,2q,r)$, and by Proposition \ref{prop:usefulcounting}, the index of this subgroup in $G$ is $[\langle \phi \rangle : \langle \phi^m \rangle] = m < \infty$. We now show that $G \cong \widetilde \Delta_k$
    
    We start by adding a redundant generator $u = (\phi^mts)^{-1}$ to $G$ to obtain the presentation
    \begin{align*}
        \langle\, s,t, u, \phi \mid s^p = t^r = e, (st)^{q} = (ts)^{q} = \phi^{-qm}, [s,\phi] = [t,\phi] = e, u = (\phi^mts)^{-1} \,\rangle.
    \end{align*}
    Define $\Phi : \widetilde \Delta_k \to G$ by 
    \begin{align*}
        \tilde a &\mapsto \phi^{k/p}s \\
        \tilde b &\mapsto \phi^{k/q}u \\
        \tilde c &\mapsto \phi^{k/r}t \\
        z &\mapsto \phi,
    \end{align*}
    then define $\Psi : G \to \widetilde \Delta_k$ by
    \begin{align*}
        s &\mapsto z^{-k/p}\tilde a \\
        u &\mapsto z^{-k/q}\tilde b \\
        t &\mapsto z^{-k/r}\tilde c \\
        \phi &\mapsto z \\
    \end{align*}
    We will show that $\Phi$ and $\Psi$ define surjective homomorphisms. Once this is shown, then clearly $\Phi$ and $\Psi$ are mutually inverse, and thus the proof of the Proposition is complete. 

    Starting with $\Phi$, we must show 
    \begin{align*}
        \Phi(\tilde a)^p = \Phi(\tilde b)^q = \Phi(\tilde c)^r = \Phi(\tilde a)\Phi(\tilde b)\Phi(\tilde c) = \Phi(z)^k
    \end{align*}
    and $\Phi(z)$ commutes with each of $\Phi(\tilde a)$, $\Phi(\tilde b)$, and $\Phi(\tilde c)$.
    First, since $\Phi(z) = \phi$ and $\phi$ is in the center of $G$, the latter relation holds. We verify
    \begin{align*}
        \Phi(\tilde a)^p 
        &= (\phi^{k/p}s)^p \\
        &= \phi^{k}s^p \\
        &= \phi^{k} \\
        &= \Phi(z)^{k},
    \end{align*}
    with an identical result for $\Phi(c)^r$. Next,
    \begin{align*}
        \Phi(\tilde b)^q 
        &= (\phi^{k/q}u)^q \\
        &= \phi^{k}u^q \\
        &= \phi^{k}(\phi^mts)^{-q} \\
        &= \phi^{k}\phi^{-qm}(ts)^{-q} \\
        &= \phi^{k}(ts)^q(ts)^{-q} \\
        &= \Phi(z)^{k}.
    \end{align*}
    Last,
    \begin{align*}
        \Phi(\tilde a)\Phi(\tilde b)\Phi(\tilde c)
        &=( \phi^{k/p}s )(\phi^{k/q}u )( \phi^{k/r}t) \\
        &= \phi^{k/p+k/q+k/r}sut\\
        &= \phi^{k/p+k/q+k/r}s(\phi^mts)^{-1}t \\
        &= \phi^{k/p+k/q+k/r-m}ss^{-1}t^{-1}t \\
        &= \phi^{k/p+k/q+k/r-(k/p + k/q + k/r - k)} \\
        &= \phi^k \\
        &= \Phi(z)^k.
    \end{align*}

    In order to show $\Psi$ is a surjective morphism, we must show
    \[
        \Psi(s)^p = \Psi(t)^r = e,\]
        \[(\Psi(s)\Psi(t))^{q} = (\Psi(t)\Psi(s))^{q} = \Psi(\phi)^{-qm},\]
    \[\Psi(u) = (\Psi(\phi)^m\Psi(t)\Psi(s))^{-1},\]
    and $\Psi(\phi)$ commutes with $\Psi(s)$ and $\Psi(t)$.
    Since $\Psi(\phi) = z$, which is in the center of $\widetilde \Delta_k$, this last relation is immediate. 
    We begin by computing
    \begin{align*}
        \Psi(s)^p
        &= (z^{-k/p}\tilde a)^p \\
        &= z^{-k}\tilde a^p \\
        &= z^{-k}z^k \\
        &= e,
    \end{align*}
    with an identical computation for $\Psi(t)^r$.
    Before proceeding, we need a lemma regarding the relations in $\widetilde \Delta_k$:
    \begin{lemma*}
        The relations $\tilde a\tilde b\tilde c = \tilde b\tilde c\tilde a = \tilde c\tilde a\tilde b$ hold in $\widetilde \Delta_k$.
    \end{lemma*}
    \begin{proof}[Proof (of Lemma)]
        Since $\tilde a\tilde b\tilde c = z^k$ and $z$ is in the center of $\widetilde \Delta_k$, we have
        \begin{align*}
            \tilde b = \tilde a^{-1}z^k\tilde c^{-1} = z^k\tilde a^{-1}\tilde c^{-1} = \tilde a^{-1}\tilde c^{-1}z^k.
        \end{align*}
        Solving each equation for $z^k$ gives $z^k = \tilde a\tilde b\tilde c = \tilde b\tilde c\tilde a = \tilde c\tilde a\tilde b$.
    \end{proof}
    For notational convenience, let $r_1 = \Psi(s)$, $r_2 = \Psi(u)$, and $r_3 = \Psi(t)$. The above lemma implies that $r_1r_2r_3 = r_2r_3r_1 = r_3r_1r_2$, and in particular $r_2$ commutes with the product $r_3r_1$. Moreover, we know $r_2$ has order $q$, since $(z^{-k/q}\tilde b)^q = z^{-k}\tilde b^q = z^{-k}z^k = 1$.
    Now,
    \begin{align*}
        (\Psi(s)\Psi(t))^{q} 
        &= (r_1r_3)^q \\ 
        &= r_1 (r_3r_1)^q r_1^{-1} \\ 
        &= r_1 r_2^q(r_3r_1)^q r_1^{-1} \\ 
        &= r_1 (r_2r_3r_1)^q r_1^{-1} \\ 
        &= (r_1r_2r_3)^q  \tag{$\ast$} \\ 
        &= (r_2r_3r_1)^q  \\ 
        &= r_2^q(r_3r_1)^q  \\ 
        &= (r_3r_1)^q \\
        &= (\Psi(t)\Psi(s))^q
    \end{align*}
    Moreover, using ($\ast$) we compute
    \begin{align*}
        (\Psi(s)\Psi(t))^{q} 
        &= (r_1r_2r_3)^q \\
        &= (z^{-k/p}\tilde az^{-k/q}\tilde b z^{-k/r}\tilde c)^q\\
        &= (z^{-k/p-k/q-k/r}\tilde a\tilde b\tilde c)^q\\
        &= (z^{-k/p-k/q-k/r}z^k)^q\\
        &= (z^{-k/p-k/q-k/r+k})^q\\
        &= (z^{-m})^q\\
        &= z^{-mq} \\
        &= \Psi(\phi)^{-mq}.
    \end{align*}
    Last,
    \begin{align*}
        (\Psi(\phi)^m\Psi(t)\Psi(s))^{-1}
        &= [z^m ( z^{-k/r}\tilde c)( z^{-k/p}\tilde a )]^{-1} \\
        &= z^{k/r+k/p-m} \tilde a^{-1}\tilde c^{-1}\\
        &= z^{k - k/q} \tilde a^{-1}\tilde c^{-1}\\
        &= z^{ - k/q} z^{k}\tilde a^{-1}\tilde c^{-1}\\
        &= z^{ - k/q} (\tilde b\tilde c\tilde a)\tilde a^{-1}\tilde c^{-1}\\
        &= z^{ - k/q} \tilde b\\
        &= \Psi(u). \qedhere
    \end{align*}
\end{proof}

\section{The syllable length condition}
\label{sec:syllable}

We now turn our attention to proving Theorem \ref{thm:cocptaction}. To do this, we follow the overarching idea used to show that the Deligne complex for a 2-dimensional Artin group is $\cat(0)$ \cite{charney1995k}. The first step in this process is to show that certain words have a minimal length. We make this precise now.

\begin{defn}
    Let $S$ be a finite set and $W(S)$ the set of (finite) words in $S \cup S^{-1}$. 
    If $w \in W(S)$ with $w = s_1^{i_1}\ldots s_n^{i_n}$ ($s_j \in S$, $i_j \in \bb{Z}$) is cyclically reduced, then we define the \emph{syllable length} of $w$ with respect to $S$ to be $\ell(w) = \ell_S(w) = n$.
\end{defn}

\begin{defn}
    If $G$ is a group with finite generating set $S$, and $w \in W(S)$, then we denote by $\overline w \in G$ the image of $w$ under the map induced by the obvious map sending the word $s_1^{i_1}\ldots s_n^{i_n}$ in $W(S)$ to the element $s_1^{i_1}\ldots s_n^{i_n}$ in $G$.
\end{defn}

The main result of this section is the following proposition, which, as mentioned before, is one of the key steps in showing the analogue of the Deligne complex for Shephard groups is also $\cat(0)$. It is based on a result of Appel and Schupp \cite{appel1983artin}, but requires a minor extra hypothesis in order to account for the torsion in the generators. 

\begin{prop} \label{prop:girth}
    Consider the dihedral Shephard group $\sh(p,q,r)$ (with $p = r$ if $q$ is odd) on standard generating set $S = \{s,t\}$ and identity element $e$. Suppose $w \in W(S)$ has a cyclically reduced expression $w = s_1^{i_1}\ldots s_n^{i_n}$, with $i_j \not\in  \bb{Z}p_{s_j}$
    (where $p_s = p$ and $p_t = r$).
    If $\overline w = e$, then $\ell(w) \geq 2q$.
\end{prop}

We will first establish notation and some brief lemmas.

\begin{defn}
    Let $(p,q,r)$ be a triple of integers all $\geq 2$ with $p = r$ if $q$ is odd. Let $Sh = \sh(p,q,r)$. 
    If $q$ is even, let $D = D(p,q,r) = \Delta(p,q/2,r)$, $\sigma = a$, and $\tau = c$, and if $q$ is odd, let $D = D(p,q,r)$ denote the subgroup of $\Delta(p,q,2)$ generated by $\sigma = a$ and $\tau = cac$. 
    Define a simplicial graph $\mathcal D = \mathcal D(p,q,r)$ whose vertices are the cosets of $\langle \sigma \rangle$ and $\langle \tau \rangle$ in $D$, with an edge between two vertices if the cosets have non-trivial intersection.
\end{defn}

Sometimes $\mathcal D$ is called a ``(rank 2) coset geometry''. It is the quotient of the complex $\widehat \Theta = \widehat \Theta(\sh(p,q,r))$ by the center of $\sh(p,q,r)$ (see Proposition \ref{prop:sphericalcplx}). 

\begin{lemma}
    For any triple $(p,q,r)$ (with $p,q,r \geq 2$ and $p = r$ if $q$ is odd), the complex $\mathcal D(p,q,r)$ is (the 1-skeleton of) a tiling of either \bb[2]{E} or \bb[2]{H} by $q$-gons. 
\end{lemma}

\begin{proof}
    The case when $q$ is even is proven in \cite{MorenoStypa}, so assume $q$ is odd and $D = \langle a, cac \rangle \leq \Delta(p,q,2)$. 
    Note that the result holds for $\mathcal D(p,2q,2)$ (since $2q$ is even and $D(p,2q,2) = \Delta(p,q,2)$) and $D(p,q,r)$ is a finite index subgroup of $D(p,2q,2)$ (by definition).
    In particular, $\mathcal D(p,2q,2)$ is a subdivision of $\mathcal D(p,q,p)$; the added vertices come from adding cosets of $c$, which correspond to midpoints of edges of $\mathcal D(p,q,p)$. In particular, since $\mathcal D(p,2q,2)$ consists of $2q$-gons and is the first barycentric subdivision of $\mathcal D(p,q,p)$, it follows that $\mathcal D(p,q,p)$ is a tiling by $q$-gons.
\end{proof}

\begin{figure}[h!]
    \centering
    \begin{multicols}{2}
        \includegraphics[width=0.45\textwidth]{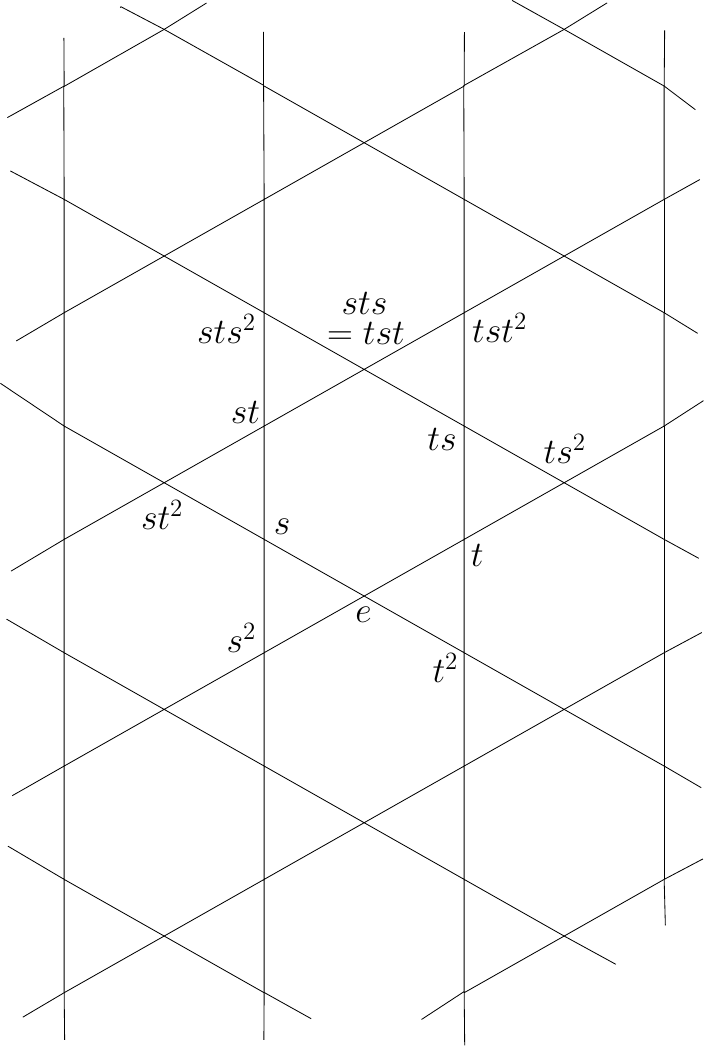}

        (a)

        \vspace{2.5em}

        \includegraphics[width=0.45\textwidth]{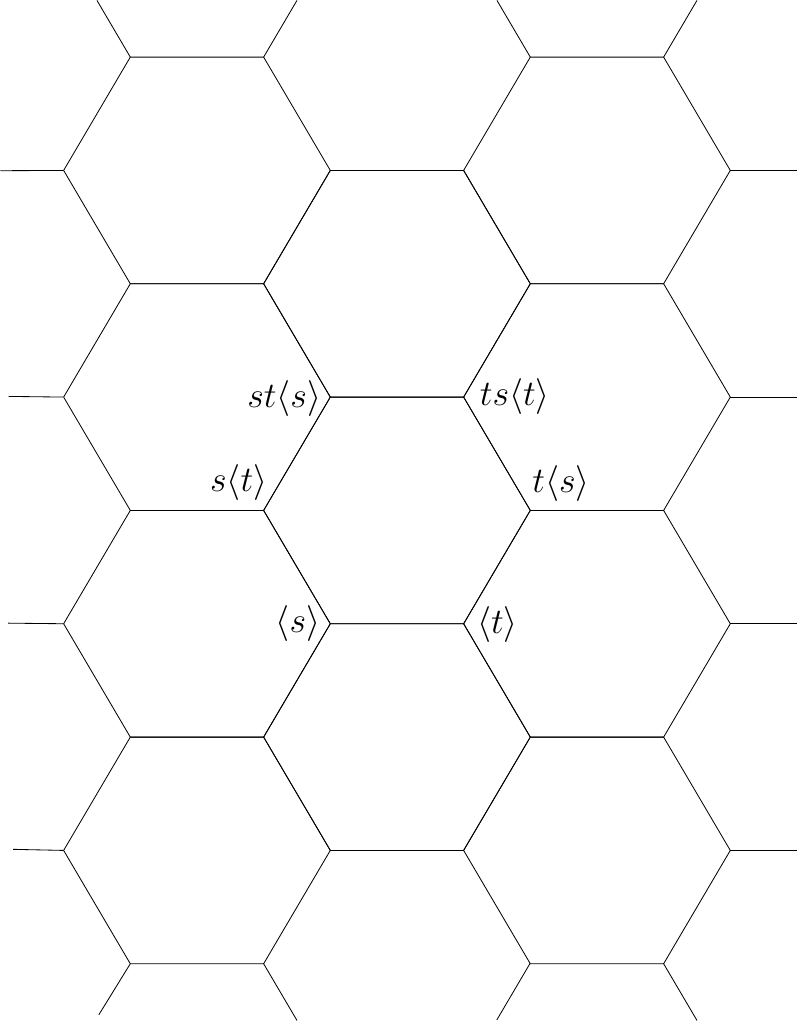}

        (b)
    \end{multicols}
    \caption{The Cayley graph (a) and coset geometry (b) for $\Delta(3,3,3)$}
    \label{fig:cayleypluscoset}
\end{figure}

The graph $\mathcal D(p,q,r)$ can be thought of as ``collapsing'' the polygons in the Cayley graph of $\Delta$ corresponding to the conjugates of the subgroups $\langle s \rangle$ and $\langle t \rangle$. For example, the Cayley graph and coset geometry of $\Delta(3,3,3)$ (coming from $\sh(3,6,3)$) are shown in Figure \ref{fig:cayleypluscoset}. They are overlaid in Figure \ref{fig:cayleycosetoverlaid} to demonstrate how the triangles induced by the orbit of $s$ and $t$ can be shrunken to their respective cosets (where these cosets are given by solid and empty vertices, respectively).

\begin{figure}[h!]
    \centering
    \includegraphics[width=0.5\textwidth]{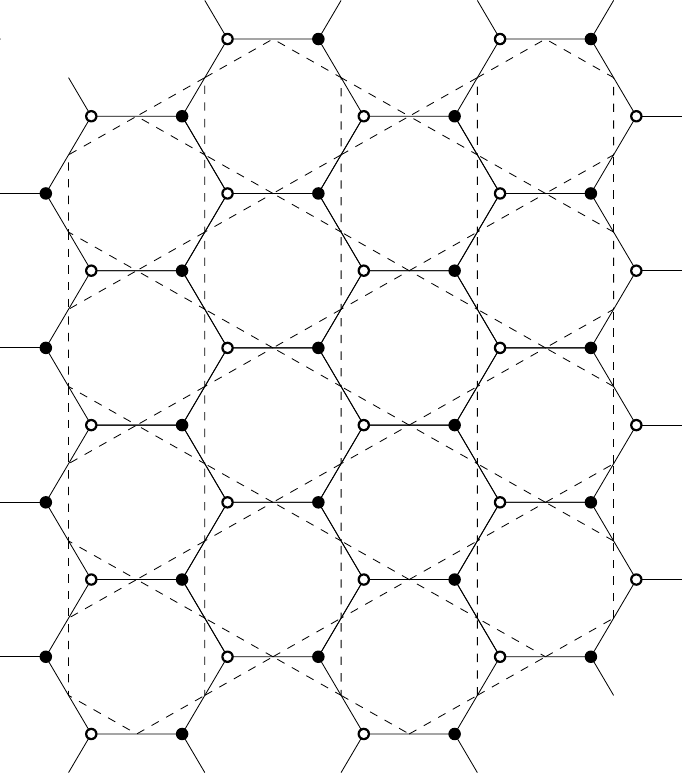}
    \caption{The Cayley graph (dashed) and coset geometry (solid) of $\Delta(3,3,3)$ overlaid}
    \label{fig:cayleycosetoverlaid}
\end{figure}

We may now prove the main Proposition of this section. The argument is based on one given in \cite[Lemma 39]{crisp2005automorphisms}.

\begin{proof}[Proof (of Prop.~\ref{prop:girth})]
    Let $E$ denote the edge of $\widehat \Theta$ coming from the intersection of the cosets $\langle s \rangle$ and $\langle t \rangle$ in $Sh$. 
    The word $w$ gives rise to a path $\gamma$ in $\widehat \Theta$ which is the concatenation $E_1E_2\cdots E_n$ of the edges $E_j$ given by
    $E_j =s_1^{i_1} s_2^{i_2}\cdots s_{j-1}^{i_{j-1}}E_{j-1}.$
     Since $i_j$ is not a multiple of $p_{s_j}$, 
    every pair of consecutive edges in this list are distinct. 
    In addition, notice that $E_n = \overline w E$. 
    So, $\gamma$ is a locally embedded closed loop, and in particular, the edge length $\ell(\gamma) = n$ of $\gamma$ is precisely the syllable length $\ell(w)$ of $w$. Without loss of generality, we may assume that this loop is embedded; otherwise, we may repeat the argument on embedded subloops. In addition, we may assume without loss of generality that $0 < i_j < p_{s_j}$ by replacing $i_j$ with its remainder after division by $p_{s_j}$; clearly this gives the same word in $Sh$ and same path in $\widehat \Theta$. In particular, this assumption does not change the syllable length.

    Let $\overline \gamma$ be the image of $\gamma$ under the covering map $\widehat \Theta \to \mathcal D$ induced by the central quotient. 
    This is still a closed loop in $\mathcal D$, but now may no longer be embedded. 
    However, can find a subpath of $\overline \gamma$ which is an embedded closed loop. After reparameterization, we write $\overline \gamma$ as the concatenation $\overline \gamma_0 \overline\gamma_1$, with $\overline\gamma_0$ an embedded closed loop in $\mathcal D$ and $\overline\gamma_1$ not necessarily embedded, possibly trivial. 
    Let $\gamma_0$ be the lift of $\overline\gamma_0$ to $\widehat \Theta$ contained in $\gamma$ and let $\gamma_1$ be the (possibly trivial) path in $\widehat \Theta$ such that $\gamma = \gamma_0\gamma_1$. 
    Then $\gamma_0$ and $\gamma_1$ represent subwords of $w$ of the form 
    $w_0 = s_1^{i_1}s_2^{i_2}\cdots s_{j}^{i_{j}}$ and
    $w_1 = s_{j+1}^{i_{j+1}}s_{j+2}^{i_{j+2}}\cdots s_{n}^{i_n}$ for some $1 \leq j \leq n$.
    In particular, $\ell(w) = \ell(\gamma) = \ell(\gamma_0) + \ell(\gamma_1)$. 
    
    Since $\overline \gamma_0$ is a non-trivial embedded loop in $\mathcal D$, it must enclose at least one $q$-gon, implying $\ell(\gamma_0) = \ell(\overline \gamma_0) \geq q$.
    We claim that $\gamma_1$ is non-empty, or in other words, that $\gamma \not= \gamma_0$. Since $\gamma$ is closed, it suffices to show that $\gamma_0$ is not a closed loop. If we show this, then, since $\overline \gamma_1$ will be a (nontrivial) closed path in $\mathcal D$, we can apply the argument given for $\overline \gamma$ and $\overline \gamma_0$ to $\overline \gamma_1$ and a simple subpath of $\overline \gamma_1$ to show that $\ell(\overline \gamma_1) \geq q$, and thus $\ell(w) \geq \ell(\overline \gamma_0) + \ell(\overline \gamma_1) \geq q + q=  2q$, as claimed.

    Showing that $\gamma_0$ is not closed is equivalent to showing $\overline{w_0} \not= e$.
    If $q$ is even, let $\Delta = \Delta(p,q/2,r)$ and if $q$ is odd, let $\Delta = \Delta(p,q,2)$ (so either $D = \Delta$ or $D$ is finite index in $\Delta$). 
    Let $K$ be the (3-skeleton of a) $K(\Delta,1)$ space defined in Section \ref{sec:centext}, with universal cover $\widetilde K$, and let $C$ be the Cayley graph of $Sh$. Note that $C$ is a covering of $\widetilde K^{(1)}$ (with $\widetilde K^{(1)}$ the Cayley graph of $\Delta$).
    The word $w_0$ gives rise to a path in $C$ in the standard way, hence also a path $\tilde \rho$ in $\widetilde K^{(1)}$ via the covering map $C \to \widetilde K^{(1)}$, and a path $\rho \in K^{(1)}$ under the covering map $\widetilde K \to K$.
    The path $\rho$ induces a cycle $\overline \rho \in C_1$ (see Section \ref{sec:centext} for notation). Since $w_0$ represents the trivial word in $\Delta$, $[\overline \rho] = 0$. 
    This means $\overline \rho \in B_1 = \im(d_2)$, so we may choose an element $R \in C_2$ such that $d_2(R) = \overline \rho$, say $R = n_a e_a + n_c e_c + n_{ac} e_{ac}$ for some $n_a,n_c, n_{ac} \in \bb{Z}$.
    Since $\gamma_0$ is simple and we have assumed $0 < i_j < p_{s_j}$, we know that $\tilde \rho$ is a simple loop, it must enclose at least one cell of $\widetilde K^{(2)}$ which maps to $e_{ac}$, and it must traverse the boundaries of each such cell with a consistent (positive) orientation. 
    This means that $n_{ac} \not= 0$, so $\varphi(R) \not= 0$. Therefore $\overline w_0$ cannot be trivial in $\sh(p,q,r)$ if $q$ is even. If $q$ is odd this shows that $\overline w_0$ is not trivial in $\sh(p,2q,2)$, but since $D$ lifts to $\sh(p,q,p) \leq \sh(p,2q,2)$ and $\overline w_0$ lies in this subgroup by assumption, the result also holds for $q$ odd.
\end{proof}

\begin{ex}
    We will give an illustrative example with the complexes $\mathcal D(p,q,r)$ and $\widehat \Theta$. (We use these over the Cayley graphs since the figures are much clearer, but a similar conceptualization works for the Cayley graphs.)

    \begin{figure}[h!]
        \centering
        \begin{multicols}{2}
        \includegraphics[width=0.375\textwidth,angle=-90]{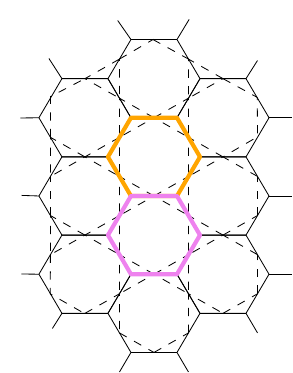}

        (a)

        \includegraphics[width=0.375\textwidth,angle=-90]{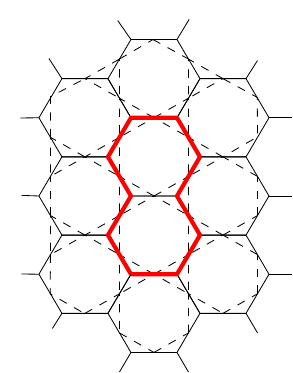}

        (b)
        \end{multicols}
        \caption{Hexagons in $\mathcal D(3,3,3)$ and a path enclosing them}
        \label{fig:hexandpathinE2}
    \end{figure}
    Consider the two hexagons in $\mathcal D(3,3,3)$ highlighted in Figure \ref{fig:hexandpathinE2}(a).
    Figure \ref{fig:hexandpathinlift}(a) shows lifts of these hexagons to the complex $\widehat \Theta$ for the Shephard group $\sh(3,6,3)$. (The dashed lines show which vertices are identified under the covering map; they are not part of $\widehat \Theta$.)
    \begin{figure}[h!]
        \centering 
        \begin{multicols}{2}
        \includegraphics[width=0.45\textwidth]{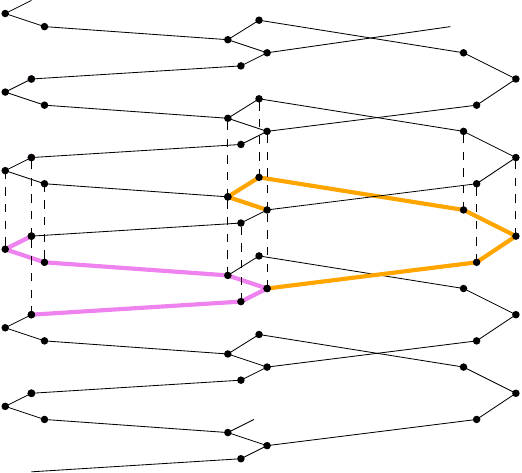}

        (a)

        \includegraphics[width=0.45\textwidth]{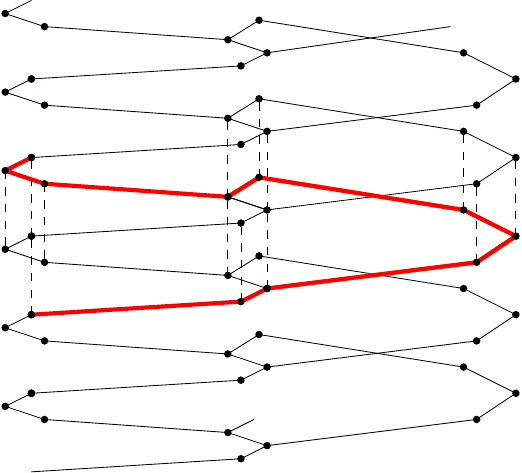}

        (b)
        \end{multicols}
        \caption{Lifts to $\widehat \Theta$ for $\sh(3,6,3)$}
        \label{fig:hexandpathinlift}
    \end{figure}
    The entire preimage of a hexagon under the covering map is a vertical column which resembles the universal cover of the circle with the cell structure coming from the hexagon.
    Consider the path encircling these two hexagons in $\mathcal D(3,3,3)$, shown in red in Figure \ref{fig:hexandpathinE2}(b). Its lift to $\widehat \Theta$ is shown in Figure \ref{fig:hexandpathinlift}(b).
    The endpoints of this path are ``distance 2'' along the fiber of the base vertex. This corresponds to the fact that the path encloses exactly two hexagons. One may compare this to the usual description of the Cayley graph of the 3-dimensional integer Heisenberg group $H(3)$, the main difference being that the vertical dashed lines would be actual edges of the Cayley graph of $H(3)$.
    This also illustrates how the center of $\sh(p,q,r)$ acts on its complex $\widehat \Theta$, since it acts by deck transformations; it is a uniform ``vertical translation'' along the dashed lines in Figure \ref{fig:hexandpathinlift} (the fibers of the map to $\mathcal D(p,q,r)$).
\end{ex}

\section{A \texorpdfstring{$\cat(0)$}{CAT(0)} cell complex for 2-dimensional Shephard groups}
\label{sec:delignecplx}

In this section, we recall the definition of $\Theta(\Gamma)$ for an arbitrary extended presentation graph $\Gamma$, which largely follows \cite[\S 3]{goldman2023cat0}.
We then show that this complex is $\cat(0)$ when $\Gamma$ is 2-dimensional.

\begin{defn} %
\label{def:complex}
Let $\Gamma$ be an extended presentation graph. 
We define $K = K_\Gamma = | (\mathcal S^{f}_\Gamma )'|$, where $(\mathcal S^{f}_\Gamma )'$ denotes the derived complex of $\mathcal S^{f}_\Gamma$ and $|(\mathcal S^{f}_\Gamma)'|$ is its geometric realization. 
We will denote an $n$-simplex of $K$ by
\[
    [\Lambda_0 < \Lambda_1 < \dots < \Lambda_n]
\] 
for a chain $\Lambda_0 < \Lambda_1 < \dots < \Lambda_n$ with each $\Lambda_i \in \mathcal S^{f}$. 
We note that the vertices are indexed by elements of $\mathcal S^{f}$; we will let $v_{\Lambda} = [\Lambda]$ denote the vertex of $K$ coming from $\Lambda$.
Define a complex of groups $\mathcal G = \mathcal G(\sh_\Gamma,K_\Gamma)$ over $K$ by declaring the local group at $v_{\Lambda}$ to be $\sh_{\Lambda}$ and the edge maps to be the natural maps coming from the inclusion of generators.

Next, define $\Delta = \Delta_\Gamma$ to be a simplex whose vertices are labeled by the generators $V(\Gamma)$ of $\sh_\Gamma$.
For $\Lambda \subseteq \Gamma$, let $\sigma_{\Lambda}$ denote the face of $\Delta_\Gamma$ spanned by the elements of $V(\Lambda)$.
We define a complex of groups $\widehat{\mathcal G} = \widehat{\mathcal G}(\sh_\Gamma, \Delta_\Gamma)$ by declaring the local group at the face $\sigma_{\Lambda}$ to be the group $\sh_{\widehat{\Lambda}}$, where $\widehat{\Lambda}$ is the full subgraph of $\Gamma$ generated by the vertices $V(\Gamma) \setminus V(\Lambda)$.
The edge maps are the standard maps induced by the inclusion of generating sets.
\end{defn}

Note that $\mathcal G$ is a simple complex of groups. Hence by \cite[Def.~II.12.12]{bridson2013metric}, the fundamental group of $\mathcal G$ is the direct limit over the edge maps; when $\Gamma$ is 2-dimensional, this direct limit is clearly $\sh_\Gamma$.
In general, neither complex of groups is a priori developable. If $\mathcal G$ is developable, we will denote its development $\Theta = \Theta(\Gamma) = \Theta_\Gamma$. If $\widehat{\mathcal G}$ is developable, we will denote its development
 $\widehat\Theta = \widehat\Theta(\Gamma) = \widehat\Theta_\Gamma$. For dihedral Shephard groups, it turns out that $\widehat\Theta$ has a straightforward description.
 
\begin{prop} \label{prop:sphericalcplx}
If $\Lambda$ is the graph which is a single edge between vertices $s$ and $t$ with all labels finite, then $\widehat{\mathcal G}(\sh_\Lambda, \Delta_\Lambda)$ is developable, and $\widehat \Theta_\Lambda$ can be described as follows: 
the vertex set of $\widehat \Theta_\Lambda$ are the cosets of $\langle s \rangle$ and $\langle t \rangle$ in $\sh_\Lambda$, and there is an edge between two vertices if the corresponding cosets intersect nontrivially. In particular, the center acts freely on $\widehat \Theta_\Lambda$.
\end{prop}
\begin{proof}
The fact that $\widehat{\mathcal G}$ is developable follows from the characterizations of $\sh(p,q,r)$ given in Section \ref{sec:centext}. Namely, in each case the cyclic groups $\bb{Z}/p\bb{Z}$ and $\bb{Z}/r\bb{Z}$ embed as the vertex groups.
The statement regarding the vertices and edges of $\widehat \Theta$ is identical to Coxeter groups and Artin groups; e.g., \cite[Proof of Lemma 4.3.2]{charney1995k}.
The center acts freely because the stabilizers of vertices and edges are finite (they are conjugates of $\langle s \rangle$ or $\langle t \rangle$ in the first case, and are trivial in the second), while the center has infinite order.
\end{proof}

When $\mathcal G$ is developable, it too has a straightforward description.

\begin{prop} \label{prop:developdescript}
    Let $\sh_\Gamma \c S^f_\Gamma = \{\, g \sh_\Lambda : \Lambda \in \c S^f_\Gamma \,\}$, ordered by inclusion. If $\mathcal G(\sh_\Gamma, K_\Gamma)$ is developable, then $\Theta_\Gamma \cong |(\sh_\Gamma \c S^f_\Gamma)'|$. In particular, the $n$-simplices of $\Theta_\Gamma$ correspond to $(n+1)$-chains $g_0 \sh_{\Lambda_0} < \dots < g_n \sh_{\Lambda_n}$ of cosets, and the action of $\sh_\Gamma$ is just the left multiplication action. We will denote the simplex arising from such a chain by
    \[
        [g_0 \sh_{\Lambda_0} < \dots < g_n \sh_{\Lambda_n}].
    \]
    The stabilizer of this simplex is $g_0^{-1} \sh_{\Lambda_0} g_0$.
\end{prop}

With this description, for developable $\mathcal G$ it is clear that $K$ embeds in $\Theta_\Gamma$ as the subcomplex consisting of simplices of the form $[ \sh_{\Lambda_0}, \sh_{\Lambda_1}, \dots, \sh_{\Lambda_n}]$. This allows one to easily see that $K$ is a fundamental domain for the action of $\sh_\Lambda$ on $\Theta_\Lambda$. In particular, we may view the vertices $v_\Lambda$ as being in $\Theta_\Gamma$ as well as $K$.
  
In order to show $\mathcal G$ is developable (under certain conditions), we show that it is ``nonpositively curved'' in the following sense.

\begin{lemma} \label{lem:locallycat0} \emph{\cite[Thm.~II.12.28]{bridson2013metric}} 
    If $\mathcal H$ is a (simple) complex of groups over a simply connected domain and the local development at each vertex is locally $\mathrm{CAT}(0)$, then $\mathcal H$ is developable and has locally $\mathrm{CAT}(0)$ development.
\end{lemma}

It is clear that $K$ is simply connected, since $v_\varnothing$ is a cone point. 
We will describe the local development shortly. We begin with putting a metric on the fundamental domain $K$.

\begin{defn}
    For $\Lambda \in \mathcal S^{f}$, let 
    \begin{align*}
        \mathcal S^{f}_{\geq \Lambda} &= \{\,\Lambda' \in \mathcal S^{f}:\Lambda' \geq \Lambda  \,\} &
         F_\Lambda &= |(\mathcal S^{f}_{\geq \Lambda})'| \\
        \mathcal S^{f}_{\leq \Lambda} &= \{\,\Lambda' \in \mathcal S^{f}:\Lambda' \leq \Lambda  \,\} &
        F_\Lambda^* &= |(\mathcal S^{f}_{\leq \Lambda})'|.
    \end{align*}
\end{defn}

Notice that $F_\Lambda^*$ is combinatorially a cube whose faces are $F_{\Lambda_1} \cap F_{\Lambda_2}^*$ where $\Lambda_1 \subseteq \Lambda_2 \subseteq \Lambda$. So $K$ itself has a cubical structure with faces $F_{\Lambda_1} \cap F_{\Lambda_2}^*$ for $\Lambda_1 \subseteq \Lambda_2 \in \mathcal S^{f}$. This Proposition is a straightforward exercise:

\begin{prop}
    With the cubical cell structure above, $F_\Lambda^*$ is isomorphic to the cone on the first barycentric subdivision $\Delta_\Lambda'$ of $\Delta_\Lambda$ with cone point $v_\Lambda$. The isomorphism is induced by the map sending $v_{\Lambda_0}$ to the barycenter of $\sigma_{\widehat{\Lambda_0}}$.
\end{prop}

Thus we may identify $lk(v_\Lambda, F_\Lambda^*)$ and $\Delta_\Lambda$.
With this connection we can define an explicit metric on $K$. 

\begin{defn} \label{defn:coxeterblock}
We give the cell $F_{\Lambda_1} \cap F_{\Lambda_2}^*$ of $K$ the metric of a \emph{Coxeter block}.
Briefly, this is (the closure of) a connected component of the Coxeter zonotope associated to the finite Coxeter group $W_{\Lambda_2}$ minus its reflection hyperplanes. 
(See \cite[\S 4.4]{charney1995k}, where this metric is defined in detail for the Davis-Moussong complex and Deligne complex.)
In this metric, if $\Lambda \in \mathcal S^{f}$, then $\Delta_\Lambda$ is a spherical simplex where the length of the edge between the vertices corresponding to $i$ and $j$ is $\pi/{m_{ij}}$. 
\end{defn}

Since $K$ is a (strict) fundamental domain, we can use the action of $\sh_\Lambda$ to metrize all of $\Theta_\Gamma$. We will call this the \emph{Moussong metric} on $\Theta_\Gamma$.

Now we recall the local development, focusing on the case of $\c G$.
Let $v_\Lambda$ be a vertex of $K$, with $\Lambda \in \mathcal S^{f}$. The \emph{upper star} $St^\Lambda$ of $v_\Lambda$ in $\mathcal G$ is the (full) subcomplex of $K$ spanned by the vertices $v_{\Lambda'}$ with $\Lambda' \geq \Lambda$. 
The \emph{lower link} $Lk_\Lambda$ of $v_\Lambda$ in $\mathcal G$ is the development of the subcomplex of groups $\widehat {\mathcal G}(K_{<\Lambda})$ of $\mathcal G(K)$, where $K_{<\Lambda}$ denotes the subcomplex spanned by vertices $v_{\Lambda'}$ with $\Lambda' \lneq \Lambda$. 
Both of these objects are simplicial complexes which inherit the metric placed on $K$.
The \emph{local development at $v_\Lambda$} is the spherical join
\begin{align*}
    D(\Lambda) = St^\Lambda \ast Lk_\Lambda.
\end{align*}
Its metric naturally comes from the metric on $K$.
The link of $v_\Lambda$ in the local development is
\begin{align*}
    lk(v_\Lambda, D(\Lambda)) &= Lk^\Lambda \ast Lk_\Lambda,
\end{align*}
where $Lk^\Lambda$ is the \emph{upper link}, meaning the (full) subcomplex of $K$ spanned by the vertices $v_{\Lambda'}$ with $\Lambda' \gneq \Lambda$. We may also sometimes refer to this complex as $K_{>\Lambda}$.

Note that $K_{>\Lambda}$ is isomorphic to $lk(v_\Lambda, F_\Lambda)$ and $K_{<\Lambda}$ is isomorphic to $lk(v_\Lambda, F^*_\Lambda)$. 
We use the previous proposition to identify $K_{<\Lambda}$ with $\Delta_\Lambda$. With this identification, the complex of groups $\widehat{\mathcal G}(K_{<\Lambda})$ is isomorphic to $\widehat{\mathcal G}(G_\Lambda, \Delta_\Lambda)$ as defined above, and thus $Lk_\Lambda$ is isomorphic to $\widehat \Theta_\Lambda$. It is straightforward to check that the metrics placed on $K$ above agree with the claimed metrics on $\Delta_\Lambda$. In other words, there is an isometry
\begin{align*}
    lk(v_\Lambda, D(\Lambda)) \cong lk(v_\Lambda, F_\Lambda) * \widehat \Theta_{\Lambda},
\end{align*}
where the join is the usual spherical join.
There are two special cases to note. When $\Lambda = \varnothing$, then $\widehat \Theta_\varnothing$ is empty, so $lk(v_\Lambda, D(\Lambda)) \cong lk(v_\Lambda, F_\Lambda)$. When $\Lambda$ is maximal in $\c S^f$, $F_\Lambda$ is a single point $v_\Lambda$, so $lk(v_\Lambda, F_\Lambda)$ is empty and $lk(v_\Lambda, D(\Lambda)) \cong \widehat \Theta_{\Lambda}$.

Showing the local development is nonpositively curved amounts to showing that these components of the links are $\mathrm{CAT}(1)$. Since a spherical join is $\mathrm{CAT}(1)$ if and only if both components are \cite[Cor.~II.3.15]{bridson2013metric}, this reduces to showing that $lk(v_\Lambda, F_\Lambda)$ and $\widehat \Theta_{\Lambda}$ are $\mathrm{CAT}(1)$ when $\Lambda \in \mathcal S^{f}$. 
With this in mind, we can now complete the proof of Theorem \ref{thm:cocptaction}. It will follow immediately by the next theorem and general facts about simple complexes of groups.

\begin{thm} \label{thm:girth}
    Suppose $\Gamma$ is a 2-dimensional extended presentation graph. Then $\c G$ is developable and its development $\Theta_\Gamma$ is $\cat(0)$. 
\end{thm}

\begin{proof}
    As in the Artin group and Coxeter group case, $lk(v_\Lambda, F_\Lambda)$ is $\cat(1)$ whenever $\Lambda \in \c S^f_\Gamma$ (see discussion before and after Lemma 4.4.1 in \cite{charney1995k}), so it remains to show that $\widehat \Theta_{\Lambda}$ is $\cat(1)$ for $\Lambda \in \c S^f_\Gamma$. Since $\Gamma$ is 2-dimensional, the only elements of $\c S^f$ are $\varnothing$, singletons, and edges. In the first two cases, $\widehat \Theta_\Lambda$ is either empty or finite, resp., so we may assume we are in the third case. Note that if $\sh_\Lambda$ is finite (i.e., has labels $1/p_i + 2/m_{ij} + 1/p_j > 1$), then this complex was shown to be $\cat(1)$ in \cite[Lemma 6.1]{goldman2023cat0}, so we may assume this group is infinite.
    
    By \cite[Lem.~II.5.6]{bridson2013metric}, it suffices to show that $\widehat \Theta_\Lambda$ has no closed loops of length $< 2\pi$. Since the length of an edge of $\widehat \Theta_\Lambda$ is $\pi/m_{ij}$, we must show that the edge length of any closed loop in $\widehat \Theta_\Lambda$ is at least $2m_{ij}$.
    Let $\gamma$ be an embedded closed loop in $\widehat \Theta_\Lambda$, say $\gamma = E_1\cdots E_n$ for edges $E_i$ of $\widehat\Theta$. Note that the edge path length $\ell(\gamma)$ is $n$. For $1 \leq i \leq n$, let $v_i$ be the vertex at which the edges $E_{i-1}$ and $E_i$ meet (with indices taken cyclically). Let $s_i$ be the generator of $\sh_\Lambda$ such that $v_i = g_i \langle s_i \rangle$ for some $g_i \in \sh_\Lambda$.
    We can write $E_i = g_is_i^{m_i}E_{i-1}$ for some $m_i \in \bb{Z}$. 
    Since $\gamma$ is embedded, $m_i$ is not a multiple of $p_{s_i}$ (the order of $s_i$).
    Stringing together these equalities gives a word $s_1^{m_1} s_2^{m_2}\cdots s_n^{m_n} = 1$ in $\sh_\Lambda$. Since $s_i \not= s_{i+1}$, the syllable length of this word is $n = \ell(\gamma$. By Proposition \ref{prop:girth}, we must have $\ell(\gamma) = n \geq 2m_{ij}$. 
    The result now follows from Lemma \ref{lem:locallycat0}.
\end{proof}

When the local groups of a nonpositively curved complex of groups are all finite, this determines all finite subgroups of the fundamental group (namely, they are the conjugates of the local groups). While we can't exactly say that here, we can determine all elements of finite order.

\begin{cor} \label{cor:finiteorderconjugate}
    Suppose $\Gamma$ is a 2-dimensional extended presentation graph. If $h \in\sh_\Gamma$ has finite order, then it is conjugate to a power of one of the standard generators of $\sh_\Gamma$.
\end{cor}

\begin{proof}
    Suppose $h \in \sh_\Gamma \setminus \{e\}$ has finite order.
    By \cite[Cor.~II.2.8(1)]{bridson2013metric}, the fixed point set $Fix(h) = \{\, x \in \Theta_\Gamma : hx = x \,\}$ is a non-empty convex subset of $\Theta_\Gamma$.
    Since $\Gamma$ is 2-dimensional, the simplices of $\Theta_\Gamma$ are at most dimension 2, and these top-dimensional simplices are of the form 
    \[
        [g_0\sh_\varnothing, g_1 \sh_{\{s\}}, g_2 \sh_e]
    \]
    for $e$ an edge of $\Gamma$ and $s$ a vertex of $e$. But the stabilizer of any point $x$ in the interior of such a cell is $g_0^{-1}\sh_\varnothing g_0$, which is the trivial group. In particular, $Fix(h)$ must be a tree in the 1-skeleton of $\Theta_\Gamma$ and avoid vertices of type $[g\sh_\varnothing]$. 
    If $Fix(h)$ contains a vertex $V = [g\sh_{\{s\}}]$ or an edge $E = [g\sh_{\{s\}}, g'\sh_e]$, then $h \in Stab(V) = g^{-1}\sh_{\{s\}}g$ and so is conjugate to a power of $s$. 
    So suppose neither of these cases occur. This implies $Fix(h) = [g\sh_e]$ for an edge $e$ of $\Gamma$ with vertices $s$ and $t$, and $h \in g^{-1} \sh_e g$. By translating, we may assume $g = e$, so $h \in \sh_e$. Corollary \ref{cor:dihedralfiniteconjugate} then implies $h$ is conjugate to a power of $s$ or $t$.
\end{proof}

\section{Acylindrical hyperbolicity}
\label{sec:acylhyp}

In \cite{vaskou2022acylindrical}, it is shown that (irreducible) 2-dimensional Artin groups of rank at least 3 are acylindrically hyperbolic. By modifying the proof appropriately, we obtain an analogous result for 2-dimensional Shephard groups as an application of Theorem \ref{thm:cocptaction}:

\acylindricallyhyperbolic*

This follows from 

\begin{prop} \label{prop:keyacylprop}
\textup{\cite[Theorem D]{vaskou2022acylindrical}} 
Let $X$ be a $\cat(0)$ simplicial complex, together with an action by simplicial isomorphisms of a group $G$. Assume that there exists a vertex $v$ of $X$ with stabilizer $G_v$ such that:
\begin{enumerate}
    \item  The orbits of $G_v$ on the link $lk(v,X)$ are unbounded, for the associated angular metric. 
    \item $G_v$ is weakly malnormal in $G$, i.e., there exists an element $g \in G$ such that $G_v \cap gG_v g^{-1}$ is finite.
\end{enumerate}
Then $G$ is either virtually cyclic or acylindrically hyperbolic.
\end{prop}

Thus the extra assumption (4) of Theorem \ref{thm:acylindricallyhyperbolic} is necessary to guarantee condition (1) of Proposition \ref{prop:keyacylprop} is satisfied; otherwise, all links would be finite graphs.

\begin{lemma} \label{lem:part1acyl}
    Let $\Gamma$ be a presentation graph satisfying the hypotheses of Theorem \ref{thm:acylindricallyhyperbolic}, and let $e$ be any edge of $\Gamma$ for which $\sh_e$ is infinite. Then the orbits of $\sh_e$ on $lk(v_e, \Theta_\Gamma)$ are unbounded.
\end{lemma}

\begin{proof}
    Since $\Theta_\Gamma$ is 2-dimensional, we know that $lk(v_e, \Theta_\Gamma) \cong \widehat \Theta_e$ (see Section \ref{sec:delignecplx}). %
    Suppose $e$ has terminal vertices $i$ and $j$, with labels $p = p_i$, $m = m_{ij}$, and $q = p_j$. Since $\sh_e$ is infinite, we know that $1/p + 2/m + 1/q \leq 1$, and $\sh_e$ is a non-trivial \bb{Z}-central extension of a finite index subgroup of a triangle group in either \bb[2]{E} or \bb[2]{H} (see Section \ref{sec:centext}). As discussed in Section \ref{sec:centext}, the  quotient of $\widehat\Theta_e$ by the center of $\sh_e$ is the 1-skeleton of a semiregular tiling of \bb[2]{E} or \bb[2]{H}, and this quotient is a covering map. In particular, the group of deck transformations (which act hyperbolically) of this cover is the central copy of \bb{Z} in $\sh_e$; %
    thus the orbits of this copy of \bb{Z} are unbounded.  
\end{proof}

In order to show that the Shephard groups in question satisfy (2) of Proposition \ref{prop:keyacylprop}, we will detail the portions of the argument which must be modified, and refer the reader to \cite[\S 5]{vaskou2022acylindrical} for the full original argument.

\begin{defn} %
    Let $\Gamma$ be an extended presentation graph. %
    For vertices $s,t$ of $\Gamma$, let $\Gamma^{st}$ denote the presentation graph with the same vertex and edge sets and labels as $\Gamma$, but with the addition of an edge $e_{st}$ labeled $6$ if there is no edge between $s$ and $t$ in $\Gamma$. (If there is already an edge between $s$ and $t$, we leave $\Gamma$ unchanged.) 
    We then define the domain $K_{\Gamma^{st}}$ and metric as in \cite[Def.~5.3]{vaskou2022acylindrical}. Our complex of groups over $K_{\Gamma^{st}}$ is defined similarly as well, but we place the free product $\bb{Z}/p_s\bb{Z} * \bb{Z}/p_t\bb{Z}$ as the edge group corresponding to $e_{st}$ if this edge was added; we make no changes if there were no changes made to $\Gamma$. This is in contrast to the Artin group case, where the rank-2 free group is placed on this edge. We denote the development of this complex of groups by $\Theta^{st}_\Gamma$.
\end{defn}

With this modification, the following key lemmas still hold, with completely identical proofs after appropriate replacements are made with the definition(s) above.

\begin{lemma} \label{lem:augmentedisCAT0} \textup{\cite[Lemma 5.6]{vaskou2022acylindrical}} 
    Let $\sh_\Gamma$ be a 2-dimensional Shephard group with $|V(\Gamma)| \geq 3$, and suppose that we are in the second case of \cite[Proposition 5.2]{vaskou2022acylindrical}. Then $\Theta^{bc}_\Gamma$ is $\cat(0)$.
\end{lemma}

\begin{lemma} \label{lem:malnormal} \textup{\cite[Lemma 5.7]{vaskou2022acylindrical}}
    Let $\sh_\Gamma$ be a 2-dimensional Shephard group with $|V(\Gamma)| \geq 3$, such that $\Gamma$ is connected and not right-angled (i.e., has some edge not labeled $2$). 
    Then there exists an edge $e$ of $\Gamma$ between vertices $a$ and $b$ with coefficient $m_{ab} \geq 3$ and an element $g \in \sh_e$ such that $\sh_e \cap g \sh_e g^{-1} = {e}$.
\end{lemma}

Now we may complete the proof of Theorem \ref{thm:acylindricallyhyperbolic}.

\begin{proof}[Proof (of Theorem \ref{thm:acylindricallyhyperbolic})]
    We may assume $\Gamma$ is connected, otherwise $\sh_\Gamma$ splits as a free product, and each free summand is an infinite group by (4); such groups are always acylindrically hyperbolic. We may assume as well that $\Gamma(\infty)$ is not right-angled; in this case, $\sh_\Gamma$ is a graph product of non-trivial (cyclic) groups, which by \cite[Cor.~2.13]{minasyan2015acylindrical} implies $\sh_\Gamma$ is either virtually cyclic or acylindrically hyperbolic. Taking $e$ to be the edge guaranteed by (4), we know $\sh_e$ is infinite, not virtually cyclic, and embeds in $\sh_\Gamma$, so under our assumptions $\sh_\Gamma$ is not virtually cyclic (hence acylindrically hyperbolic).
    The proof of \cite[Prop.~5.2]{vaskou2022acylindrical} and Lemma \ref{lem:malnormal} imply that we may choose the edge $e$ in Lemma \ref{lem:malnormal} to be the edge $e$ guaranteed by (4). We note that this edge has label $\geq 3$, since otherwise $\sh_e$ would be a direct product of finite groups and hence finite. So Lemmas \ref{lem:part1acyl} and \ref{lem:malnormal} imply $\sh_\Gamma$ satisfy the hypotheses of Proposition  \ref{prop:keyacylprop}, and hence is acylindrically hyperbolic.
\end{proof}

\section{Relative hyperbolicity and residual finiteness}
\label{sec:relhyp}

Recall the following characterization of relative hyperbolicity, due to Bowditch \cite{bowditch2012relatively}. This phrasing comes from \cite[Def.~6.2]{osin2006relatively}.

\begin{prop} \label{prop:bowditchequiv}
    Let $G$ be a group and $\c P = \{P_1,\dots,P_n\}$ a collection of infinite finitely generated subgroups (we will call $(G, \c P)$ a \emph{group pair} and the elements of $\c P$ the \emph{peripheral subgroups}). Then $(G, \c P)$ is relatively hyperbolic if and only if $G$ admits an action on a hyperbolic graph $Y$ such that each of the following hold.
    \begin{enumerate}
        \item All edge stabilizers are finite.
        \item All vertex stabilizers are either finite or conjugate to one of the subgroups of $\c P$.
        \item The number of orbits of edges is finite.
        \item The graph $Y$ is \emph{fine}.
    \end{enumerate}
\end{prop}

There are a number of equivalent definitions of a fine graph, but the one that shall be useful for us is:

\begin{defn} \cite[Def 2.1.(F5)]{bowditch2012relatively}
    Let $Y$ be a graph and
    let $d^Y$ denote the metric on $Y$ which assigns all edges length $1$. (By convention, for any graph $Y$ with such a metric, if $y$ and $z$ are in different connected components of $Y$, we say $d^Y(y,z) = \infty$.) 
    Fix a vertex $x \in V(Y)$ and let $Y \setminus x$ denote the largest full subgraph of $Y$ which avoids $x$. Let $V_x(Y)$ be the set of vertices of $Y$ which are adjacent to $x$.
    Let $d^{Y \setminus x}$ denote the induced length metric on $Y \setminus x$ coming from $d^Y$, and let $d_x$ denote the restriction of the metric $d^{Y \setminus x}$ to $V_x(Y)$ (\emph{not} the induced length metric).
    We say that $Y$ is \emph{fine} if the metric space $(V_x(Y), d_x)$ is locally finite\footnote{Every finite-radius ball is a finite set.} for every $x \in V(Y)$.
\end{defn}

In order to show that (certain) Shephard groups $\sh_\Gamma$ are relatively hyperbolic, we will use the action of $\sh_\Gamma$ on its complex $\Theta_\Gamma$. Specifically, we will let $Y = Y_\Gamma$ denote the 1-skeleton of $\Theta_\Gamma$ endowed with the edge-path metric (each edge is given length $1$). When $\Theta_\Gamma$ is Gromov hyperbolic, so too is its 1-skeleton as a metric graph under the induced length metric. But this metric graph is quasi-isometric to $Y$, and hence $Y$ is also hyperbolic. So, our first task is to determine when $\Theta_\Gamma$ is hyperbolic.

\begin{lemma} \label{lem:hypimplieshyp}
    Let $\Gamma$ be a 2-dimensional extended presentation graph. If $W_\Gamma$ is word hyperbolic, then $\Theta_\Gamma$ is Gromov hyperbolic.
\end{lemma}
\begin{proof}
    By Theorem \ref{thm:girth}, $\Theta_\Gamma$ is $\cat(0)$. So by the Flat Plane Theorem, $\Theta_\Gamma$ is Gromov hyperbolic if and only if it contains no isometrically embedded copy of \bb[2]{E}. 
    Suppose such an embedded plane exists (so $\Theta_\Gamma$ is not hyperbolic).
    This plane must be a subcomplex of $\Theta_\Gamma$, and in particular must pass through a vertex $v$ of the form $[g \sh_\varnothing]$ for some $g \in \sh_\Lambda$.
    This gives rise to an embedded loop of length exactly $2\pi$ in the link of $v$. 
    The link of such a vertex is isometric to $lk(v_\varnothing,F_\varnothing) \ast \widehat \Theta_\varnothing \cong lk(v_\varnothing,K)$ (see discussion after Definition \ref{defn:coxeterblock}).
    As in the Artin group case, such a link contains a circuit of length exactly $2\pi$ if and only if $W_\Gamma$ is not hyperbolic \cite[Proof of Lemma 5]{crisp2005automorphisms}. In summary, if $W_\Gamma$ is hyperbolic, then $\Theta_\Gamma$ has no embedded flat plane, and thus $\Theta_\Gamma$ is Gromov hyperbolic.
\end{proof}

We would like to say that this is an ``if and only if'' statement, as in the Artin group setting \cite{charney2007relative}, but are unsure how to proceed. For the Artin groups, this relies on the existence of an embedding of the Davis complex in the Deligne complex, which is not yet known to exist for arbitrary Shephard groups. However, if such an embedding exists, then the reverse implication is immediate.

\begin{prop}
    Suppose $\Gamma$ is a 2-dimensional extended presentation graph such that the Davis complex $\Sigma_\Gamma$ for the Coxeter group $W_\Gamma$ embeds isometrically in the complex $\Theta_\Gamma$. Then if $\Theta_\Gamma$ is Gromov hyperbolic, $W_\Gamma$ must be word hyperbolic.
\end{prop}

\begin{proof}
    Suppose $W_\Gamma$ is not hyperbolic. Then $\Sigma_\Gamma$ is not hyperbolic (since $W_\Gamma$ is quasiisometric to $\Sigma_\Gamma$), so it contains an embedded flat plane. Since $\Sigma_\Gamma$ embeds isometrically in $\Theta_\Gamma$, we have that $\Theta_\Gamma$ also contains an embedded flat plane, so is not hyperbolic.
\end{proof}

The existence or non-existence of such an embedding is outside our current scope of consideration. (Although it is natural to conjecture that there is always such an embedding, since it exists for Artin groups and can be constructed case-by-case for finite Shephard groups.)

Now we show that $Y$ is fine. First recall the following notation. 
    If $X$ is a cell complex and $x \in X$, the open star $st(x) = st(x,X)$ is the union of all open cells containing $x$, the closed star $St(x) = St(x,X)$ is the (topological) closure of the open star, and the boundary of the closed star is $\partial St(x) = St(x) \setminus st(x)$. 

\begin{lemma} \label{lem:combstays}
    Suppose $\Gamma$ is 2-dimensional. %
    Let $x$ be a vertex of $\Theta_\Gamma$ (equivalently, of $Y_\Gamma$) of the form $[g \sh_e]$ for $g \in \sh_\Gamma$ and an edge $e$ of $\Gamma$.
    Let $\ell_x$ denote the induced length metric on $\partial St(x) \subseteq \Theta_\Lambda$. Then there is a $C \geq 1$ and $D \geq 0$ such that $\ell_x(y,z) \leq Cd_x(y,z) + D$ for all $y,z \in V_x(Y)$. 
\end{lemma}

\begin{proof}
    Let $\iota : Y_\Gamma \to \Theta_\Gamma$ denote the inclusion map which realizes the quasi-isometry of $Y_\Gamma$ with the edge length metric and $\Theta_\Gamma$ with the Moussong metric. 
    The restriction of $\iota$ to $Y \setminus x$ is a quasi-isometry onto $\Theta_\Gamma \setminus st(x)$, under the respective induced length metrics $d^{Y\setminus x}$ and $d^{\Theta_\Gamma \setminus st(x)}$. 
    We choose our $C$ and $D$ to be the constants from this restricted quasi-isometry, i.e., those constants which satisfy
    \[
        \frac{1}{C} d^{Y\setminus x}(y,z) - D \leq d^{\Theta_\Gamma \setminus st(x)}(y,z) \leq C d^{Y\setminus x}(y,z) + D
    \]
    for all $y,z \in Y \setminus \{x\}$.  Recall that $d_x$ is the restriction of $d^{Y \setminus x}$ to $V_x(Y)$, so if we restrict ourselves to elements $y,z$ of $V_x(Y)$, then we can say  
    \[
        \frac{1}{C} d_x(y,z) - D \leq d^{\Theta_\Gamma \setminus st(x)}(y,z) \leq C d_x(y,z) + D
    \]
    
    Fix $y,z \in V_x(Y)$. Let $\gamma$ be a geodesic in $\Theta_\Gamma \setminus st(x)$ from $y$ to $z$. Then
    \[
        \ell(\gamma) = d^{\Theta_\Gamma \setminus st(x)}(y,z) \leq C d_x(y,z) + D.
    \]
    Let $\rho : \Theta_\Gamma \to St(x)$ be the closest point projection onto the closed star $St(x)$ (the map which sends a point $p \in X$ to a point $\rho(p) \in St(x)$ which minimizes $d^{\Theta_\Gamma}(p, \rho(p))$). 
    Note that $St(x)$ is a convex set: it is made up of Euclidean right triangles which have one of their acute angles meeting at the common vertex $x$. Thus $\rho$ is a well-defined, distance non-increasing retraction \cite[Prop.~II.2.4(4)]{bridson2013metric}, and in particular restricts to a well-defined distance non-increasing retraction 
    $\rho : \Theta_\Gamma \setminus st(x) \to \partial St(x)$. 
    So $\rho(\gamma)$ is a (rectifiable) path in $\partial St(x)$ between $y$ and $z$, implying
    \[
        \ell_x(y,z) \leq \ell(\rho(\gamma)) \leq \ell(\gamma) \leq C d_x(y,z) + D. \qedhere
    \]

\end{proof}

We can now complete the proof of Theorem \ref{thm:relativelyhyperbolic}.

\relativelyhyperbolic*

\begin{proof} %
    When $W_\Gamma$ is hyperbolic, then $\Theta_\Gamma$ is hyperbolic by Lemma \ref{lem:hypimplieshyp}. In particular, $Y$ is a hyperbolic graph since it is quasi-isometric to the 1-skeleton of $\Theta_\Gamma$. By Proposition \ref{prop:developdescript}, the edge stabilizers of $\sh_\Gamma$ acting on $Y$ are either finite cyclic (coming from the subgroups generated by the vertices) or trivial. Similarly, the vertex stabilizers are the conjugates of $\sh_\Lambda$ for $\Lambda \in \c S^f$. And since this is a cocompact action with a strict fundamental domain, there are finitely many orbits of edges. Once we show that $Y$ is fine, the result will follow by Proposition  \ref{prop:bowditchequiv}.

    Since the action of $\sh_\Gamma$ on $Y$ has a strict fundamental domain, we may restrict our consideration to vertices of the form $v_\Lambda$ for $\Lambda \in \c S^f$ (using notation from Section \ref{sec:delignecplx}).
    Since $\Gamma$ is 2-dimensional, there are three types of vertices $x = v_\Lambda$ to consider: $\Lambda = \varnothing$, $\Lambda = \{s\}$ a single vertex, or $\Lambda = e$ a single edge.
    If $\Lambda = \varnothing$, then, as a set, $V_x(Y)$ is simply the vertex set of $lk(v_\varnothing,F_\varnothing)$, which is finite. Similarly, if $\Lambda = \{s\}$, then $V_x(Y)$ is the vertex set of the join of $lk(v_{\{s\}},F_{\{s\}})$ and $\widehat \Theta_{\{s\}}$, which are both finite, and hence $V_x(Y)$ is finite. 

    Suppose $\Lambda = e$ is an edge between vertices $s$ and $t$ of $\Gamma$. 
    Since $\Theta_\Gamma$ is a piecewise Euclidean simplicial complex with finitely many isometry types of faces, $lk(x,\Theta_\Gamma)$ is isometric to a sufficiently small sphere centered at $x$, and this sphere is a radial deformation retract of $\partial St(x)$. This graph isomorphism is actually a quasi-isometry. 
    We also know that $\widehat \Theta_e$ with the Moussong metric is quasi-isometric to $\widehat \Theta_e$ with the metric $d^{\widehat \Theta_e}$ which assigns all edges length $1$. 
    In particular, there are $C' \geq 1$ and $D' \geq 0$ such that $d^{\widehat \Theta_e}(y,z) \leq C' \ell_x(y,z) + D'$ for all $y,z \in \partial St(x)$ (or $\widehat \Theta_e$). 
    
    Now, let $C$ and $D$ be the constants guaranteed by Lemma \ref{lem:combstays}. 
    Fix a vertex $y$ of $\partial St(x)$ (or $V_x(Y)$). For $N  > 0$, define
    \begin{align*}
        V_N &= \{\, z \in V_x(Y) : d_x(y,z) \leq N \,\}, \\
        L_N &= \{\, z \in \widehat \Theta_e : d^{\widehat \Theta_e}(y,z) \leq C'(CN + D) + D' \,\}.
    \end{align*}
    Let $z \in V_N$. Then 
    \[
        d^{\widehat \Theta_e}(y,z) \leq C' \ell_x(y,z) + D' \leq C'( C d_x(y,z) + D) + D' \leq C'(CN + D) + D'.
    \]
    $\ell_x(y,z) \leq C d_x(y,z) + D \leq CN + D$, so $z \in S_N$. Hence $V_N \subseteq L_N$ for all $N > 0$. 
    We know that $\widehat\Theta_e$ is a locally finite graph: the vertices of $\widehat \Theta_e$ coming from cosets of $\langle s \rangle$ have valence $p_s$ and the vertices coming from cosets of $\langle t \rangle$ have valence $p_t$, both of which we have assumed to be finite. %
    This implies $L_N$ is finite for every $N$, and hence so is $V_N$. Therefore $V_K(x)$ is locally finite and $Y_\Gamma$ is fine.
\end{proof}

This result is in stark contrast to the Artin group case, where it is very uncommon for Artin groups to be relatively hyperbolic rather than just weakly relatively hyperbolic. We can leverage this not only to show many nice properties of Shephard groups, but also of many Artin groups as well. In the following section, we will discuss an application to Artin groups, but first we will detail some easy consequences of relatively hyperbolicity for Shephard groups. 

\relhypconsequences*

\begin{proof}
    Each property in the list holds for all of $\sh_\Gamma$ if and only if it holds for the peripheral subgroups. In more detail:
    \begin{enumerate}
        \item The dihedral Shephard groups are linear, hence have solvable word problem. This implies $\sh_\Gamma$ has solvable word problem by \cite{farb1998relatively}. 
        
        \item If $(G, \c P)$ is a relatively hyperbolic pair, then it was shown in \cite{tukia1994convergence} that any subgroup of $G$ which does not contain a free group is either finite, virtually infinite cyclic, or is contained in an element of $\c P$. 
        Since the peripheral subgroups of $\sh_\Lambda$ are linear, they satisfy the Tits alternative, so we can conclude $\sh_\Lambda$ does as well. 
        
        \item By \cite{bell2008asymptotic}, asymptotic dimension of finitely generated groups is preserved by commensurability. Since the 3-dimensional integral Heisenberg group and the universal central extensions of surface groups have finite asymptotic dimension, so too do the peripheral subgroups of $\sh_\Lambda$. By \cite{osin2005asymptotic}, this implies $\sh_\Gamma$ has finite asymptotic dimension.
        
        \item 
        If $e = \{i,j\}$ is an edge of $\Gamma$ with $1/p_i + 2/m_{ij} + 1/p_j = 1$, then $\sh_e$ has polynomial growth since it is commensurable to the 3-dimensional integral Heisenberg group; groups of polynomial growth have the rapid decay property by \cite{jolissaint1990rapidly}.
        If $e = \{i,j\}$ is an edge of $\Gamma$ with $1/p_i + 2/m_{ij} + 1/p_j < 1$, then $\sh_e$ is a \bb{Z}-central extension of a hyperbolic group; such groups have the rapid decay property by \cite{noskov1992algebras}.
        This implies $\sh_\Gamma$ has the rapid decay property by \cite{dructu2005relatively}.
    \end{enumerate}

    Last, if each edge $\{i,j\}$ of $\Gamma$ satisfies $1/p_i + 2/m_{ij} + 1/p_j \not= 1$, then each peripheral subgroup is biautomatic by Proposition \ref{prop:negbiauto}. This implies $\sh_\Lambda$ is biautomatic by \cite{rebbechi2001algorithmic}. 
\end{proof}

A more substantial corollary is residual finiteness for certain 2-dimensional Shephard groups and their Artin groups.
To discuss residual finiteness, we begin by recalling the notion of a relatively geometric action recently introduced by Einstein and Groves.

\begin{defn} \cite[Def 1.1]{einstein2022relatively}
    Suppose $(G, \c P)$ is a group pair. An action of $G$ on a cell complex $X$ is relatively geometric (with respect to $\c P)$ if
    \begin{enumerate}
         \item $X / G$ is compact,
         \item Each group in $\c P$ acts elliptically on $X$, and
         \item Each stabilizer of a cell in $X$ is either finite, or conjugate to a finite index subgroup of an element of $\c P$.
     \end{enumerate} 
\end{defn}

It is clear from Proposition \ref{prop:developdescript} that if $G = \sh_\Gamma$ is a 2-dimensional Shephard group and $\c P$ is as defined previously, then the action of $\sh_\Gamma$ on its complex $\Theta_\Gamma$ is relatively geometric with respect to $\c P$. We want to make use of:

\begin{prop} \label{prop:residcondition} \textup{\cite[Cor 1.7]{einstein2022relatively}}
    Suppose $(G, \c P)$ is relatively hyperbolic and acts relatively geometrically (with respect to $\c P$) on a $\cat(0)$ cube complex $X$. If every $P \in \c P$ is residually finite, then $G$ is residually finite.
\end{prop}

We have a characterization of when $\sh_\Gamma$ is relatively hyperbolic; if we can determine when $\Theta_\Gamma$ is a $\cat(0)$ cube complex, we will determine a class of residually finite Shephard groups.

\begin{lemma} \label{lem:2dfc}
    If $\Gamma$ is 2-dimensional and type FC, then $\Theta_\Gamma$ is a $\cat(0)$ cube complex under the ``cubical metric''. 
\end{lemma}

\begin{proof}
    The cubical metric on the complex $\Theta_\Gamma$ is defined as follows: rather than metrize the cell $F_{\Lambda_1} \cap F_{\Lambda_2}^*$ to be a Coxeter block as in Definition \ref{defn:coxeterblock}, we simply give it the metric of a standard Euclidean cube $[0,1]^n$ (since this cell is combinatorially a cube).
    Under this metric, the link of a vertex $v_\Gamma$ is still isometric to the spherical join of $lk(v_\Lambda, F_\Lambda)$ and $\widehat \Theta_\Lambda$, where the metric on these complexes now assigns edge lengths of $\pi/2$.
     
    In order to show that $\Theta_\Gamma$ is a $\cat(0)$ cube complex, we must show that the links are flag complexes. Since the spherical join of flag complexes is flag, this is equivalent to showing $lk(v_\Lambda, F_\Lambda)$ and $\widehat \Theta_\Lambda$ are flag for all $\Lambda \in \c S^f$. This is well known for $lk(v_\Lambda, F_\Lambda)$ since $\Gamma$ is type FC \cite[Lemma 4.3.4]{charney1995k}. If $\Lambda = \varnothing$ or a single vertex $\{s\}$, then $\widehat \Theta_\Lambda$ is either empty or a finite set, resp., so there is nothing to check. So, suppose $\Lambda$ is an edge between vertices $s$ and $t$. Since $\widehat \Theta_\Lambda$ is a graph (i.e., a 1-dimensional simplicial complex), it suffices to show that it contains no 3-cycles. But by Theorem \ref{thm:girth}, the girth of this graph is $\geq 2m_{st} \geq 4$. 
\end{proof}

Thus we may conclude,

\residuallyfinite*

We note that the condition that $\Gamma$ is 2-dimensional and type FC is equivalent to requiring $\Gamma$ has no 3-cycles (sometimes called ``triangle-free''). For a triangle-free presentation graph, having no 4-cycles with all edges labeled $2$ is equivalent to $W_\Gamma$ being hyperbolic (originally due to Moussong in \cite{moussong1988hyperbolic}, rephrased in terms of the presentation graph in \cite[Lemma 5]{crisp2005automorphisms} or \cite[Prop.~3.1]{charney2007relative}). 

\begin{proof}
    Since $W_\Gamma$ is hyperbolic, $\sh_\Gamma$ is hyperbolic relative to its infinite spherical-type edge subgroups (Theorem \ref{thm:relativelyhyperbolic}). Since $\Gamma$ is 2-dimensional and type FC, $\Theta_\Gamma$ is a $\cat(0)$ cube complex under the cubical metric (Lemma \ref{lem:2dfc}) and the action is relatively geometric with respect to the spherical-type edge subgroups. Since the dihedral Shephard groups are residually finite (Proposition \ref{prop:edgesresidfinite}), we conclude that $\sh_\Gamma$ is residually finite (Proposition \ref{prop:residcondition}).
\end{proof}

In tandem with Corollary \ref{cor:finiteorderconjugate}, we can show

\begin{cor}
    If $\Gamma$ is triangle-free and has no 4-cycle with all labels $2$, then $\sh_\Gamma$ is virtually torsion-free.
\end{cor}

\begin{proof}
    Let $s \in V(\Gamma)$ be a standard generator of $\sh_\Gamma$. Since $\sh_\Gamma$ is residually finite, there is a finite index subgroup $G_s < \sh_\Gamma$ which avoids $s$ (meaning $s \not\in G_s$), and by taking finitely many intersections, we can choose $G_s$ to avoid all nontrivial powers of $s$. 
    Let $G = \bigcap_{s \in V(\Gamma)} G_s$, a finite index subgroup of $\sh_\Gamma$ avoiding all nontrivial powers of standard generators.
    Let $N = \bigcap_{g \in \sh_\Gamma} g^{-1} G g < G$. It is a standard exercise to show that $N$ (the ``normal core'' of $G$) is finite index when $G$ is finite index.
    Suppose $g \in N$ has finite order. 
    By Corollary \ref{cor:finiteorderconjugate}, we can write $g = h^{-1}s^kh$ for some $h \in \sh_\Gamma$, some $s \in V(\Gamma)$, and some $k \in \bb{Z}$. The definition of $N$ implies $g \in h^{-1}Gh$, so $s^k \in G$. Since the only power of a generators contained in $G$ is $e$, we must have that $s^k = e$, and consequently $g = e$. Thus $N$ is torsion-free.
\end{proof}

\section{Application to Artin groups}
 \label{sec:artinresid}

We now establish the analogue of Corollary \ref{cor:residuallyfinite} for Artin groups.
First, we recall the definition of the Deligne complex of an Artin group to establish the notation and terminology which we will use, as it may differ from some references.

\begin{defn}
    Let $\Gamma$ be any presentation graph and $A_\Gamma \c S^f = \{\,a A_\Lambda : a \in A_\Gamma, \Lambda \in \c S^f\,\}$ ordered by inclusion. 
    The \emph{Deligne complex} $\Phi_\Gamma$ of $A_\Gamma$ is $|(A_\Gamma \c S^f)'|$, the geometric realization of the derived complex $(A_\Gamma \c S^f)'$.  An $n$-simplex of $\Phi_\Gamma$ is written as 
    \begin{align*}
        [ a_0 A_{\Lambda_0} < a_1 A_{\Lambda_1} < \dots < a_n A_{\Lambda_n} ]
    \end{align*} 
    where $a_0 A_{\Lambda_0} < a_1 A_{\Lambda_1} < \dots < a_n A_{\Lambda_n}$ is a chain of elements of $A_\Gamma \c S^f$. For $\Lambda \in \c S^f_\Gamma$, the \emph{spherical Deligne complex} $\widehat \Phi_\Lambda$ is the simplicial complex whose vertices are all cosets of $A_{\widehat s}$ for $s \in \Lambda$, with a set of $n+1$ vertices spanning an $n$-simplex if and only if they have nontrivial (global) intersection\footnote{By $\widehat s$, we mean the largest full subgraph of $\Lambda$ which does not contain $s$.}.
\end{defn}

$A_\Gamma$ acts on $\Phi_\Gamma$ with a strict fundamental domain isomorphic to $K_\Gamma$. Thus we may endow the fundamental domain of $\Phi_\Gamma$ with the ``same'' metric as $\Theta_\Gamma$ (the Moussong metric). For the vertex $v_\Lambda = [A_\Lambda]$ of $\Phi_\Gamma$, the link $lk(v_\Lambda, \Phi_\Gamma)$ is isometric to the spherical join
\[
    lk(v_\Lambda, F_\Lambda) * \widehat\Phi_\Lambda
\] 
under this metric \cite{charney1995k}, where $F_\Lambda$ is defined as in Section \ref{sec:delignecplx}.

One of the key properties which will allow us to pass between an Artin group and its Shephard quotients is product separability.
The following is one of the many equivalent notions of product separability.

\begin{defn}
    Let $G$ be any group. We say that $G$ is \emph{product separable} if for any (finite) collection $\{G_1,\dots,G_n\}$ of finitely generated subgroups of $G$,
    their product $G_1G_2\cdots G_n$ is closed in the profinite topology on $G$.
\end{defn}

Recall that the profinite topology on $G$ is the topology whose basis of closed sets are the finite-index subgroups of $G$. So it is equivalent to say that $G_1G_2\cdots G_n$ is an intersection of finite-index subgroups of $G$. 

\begin{lemma}\label{lem:virtuallyproductseparable}
    Let $G$ be any group and $H$ a finite-index subgroup.
    Then $G$ is product separable if and only if $H$ is product separable.
\end{lemma}
\begin{proof}
    Suppose $G$ is product separable.
    Let $\{H_1,\dots,H_n\}$ be a collection of finitely generated subgroups of $H$. %
    The inclusion map $H \hookrightarrow G$ is continuous for \emph{any} subgroup; the intersection of $H$ with a finite index subgroup of $G$ is finite index in $H$ regardless of the index of $H$ in $G$. In particular, $H_1H_2 \cdots H_n$ is closed in $G$ and is contained in $H$, so $H_1H_2 \cdots H_n$ is closed in $H$.  

    Now suppose $H$ is product separable.
    By replacing $H$ with the normal core of $H$, we may assume $H$ is normal in $G$. (When $H$ is finite index, so is its normal core, and product separability is inherited by subgroups.)
    Now let $\{G_1,\dots,G_n\}$ be a collection of finitely generated subgroups of $G$, and let $H_i = G_i \cap H$ for each $i$.
    Each $H_i$ is finitely generated since the $G_i$ are, 
    and $H_i$ is finite index in $G_i$ for each $i$. So, for each $i$, let $\{h_1^{(i)}, \dots, h_{k_i}^{(i)}\}$ be a system of representatives of the left cosets of $H_i$ in $G_i$.
    For $j = (j_1,\dots,j_n)$, let $S_j = h_{j_1}^{(1)} H_1h_{j_2}^{(2)} H_2 \dots h_{j_n}^{(n)} H_n$.
     Then 
    \begin{align*}
        G_1G_2\cdots G_n = \bigcup_{j =(j_1,\dots,j_n)} S_j.
    \end{align*}
    We claim that each $S_j$ is closed in $G$. For convenience, fix $j$ and let $h_i = h_{j_i}^{(i)}$. 
    Then we can rewrite $S_j$ as
    \begin{align*}
        S_j = (h_1 H_1h_1^{-1})(h_1h_2 H_2 (h_1h_2)^{-1}) (h_1h_2h_3H_3(h_1h_2h_3)^{-1}) \dots & \\ \dots (h_1h_2\cdots h_{n-1} H_{n-1} (h_1h_2\cdots h_{n-1} )^{-1})(h_1h_2\cdots h_n H_n).&
    \end{align*}
    Since $H$ is normal in $G$,
    \begin{align*}
        S_j (h_1h_2\cdots h_n )^{-1} = (h_1 H_1h_1^{-1})(h_1h_2 H_2 (h_1h_2)^{-1}) (h_1h_2h_3H_3(h_1h_2h_3)^{-1}) \dots & \\ \dots (h_1h_2\cdots h_n H_n (h_1h_2\cdots h_n )^{-1})&
    \end{align*}
    is the product of finitely generated subgroups of $H$, thus closed in $H$.
    Since $H$ is finite index in $G$, it is closed in $G$ as well. 
    The profinite topology is invariant under multiplication (from the left or right), so $S_j$ is also closed in $G$. Thus $G_1G_2\cdots G_n$ is the union of finitely many closed subsets of $G$, so is itself closed.
\end{proof}

\begin{cor} \label{cor:dihedralartinsep}
    Dihedral Artin groups are product separable.
\end{cor}

\begin{proof}
    Suppose $A$ is a dihedral Artin group with edge label $m$. If $m = 2$, then $A \cong \bb[2]{Z}$, and product separability in this case is an easy exercise (following directly from the fact that abelian groups are subgroup separable, i.e., every f.g.~subgroup is closed). If $m > 2$, then $A$ is virtually $\bb{Z} \times \bb F_n$, where $\bb F_n$ is a rank-$n$ free group \cite[Lemma 4.3]{huang2016cocompactly}. By \cite{you1997product}, $\bb{Z} \times \bb F_n$ is product separable for any $n$, so by Lemma \ref{lem:virtuallyproductseparable}, so is $A$. 

    By \cite[Thm.~2.1]{ribes1993profinite}, the rank-2 free group $\bb{F}_2$ is also product separable; this is the dihedral Artin group whose graph is two vertices not joined by an edge.
\end{proof}

\begin{lemma} \label{lem:prodsepimpliesfinitequotient}
    If $G$ is any group and $C$ is closed in the profinite topology on $G$, then for all $g \not\in C$, there exists a finite group $F$ and a surjective morphism $\varphi : G \to F$ such that $\varphi(g) \not\in \varphi(C)$. 
\end{lemma}

\begin{proof}
    Since $C$ is closed in the profinite topology, it is the intersection of the finite-index subgroups of $G$ containing it. At least one of these subgroups, say $H$, must avoid $g$. Let 
    $K$ 
    denote the normal core of $H$, a finite-index normal subgroup of $G$ contained in $H$. 
    Let $\varphi$ denote the standard quotient map $G \to G/K$. 
    Then $gK \not\subset CK$ since $CK \subseteq H$, implying $\varphi(g) \not\in \varphi(C)$.
\end{proof}

The following lemma is the main ingredient in the proof of residual finiteness for these Artin groups.

\begin{lemma} \label{lem:dihedralartinquotient}
    Suppose $\Lambda = e$ is a single edge with label $q < \infty$. Let $x,y \in \widehat \Phi_\Lambda$ be any two points.
    Then there is some $N \in \bb{Z}_{> 0}$ such that for all $k \geq 1$, the images of $x$ and $y$ under the quotient map $\widehat \Phi_\Lambda \to \widehat \Theta_{\Lambda(kN)}$ remain at the same distance.
\end{lemma}

\begin{proof}
    Suppose $s$ and $t$ are the vertices of $e$, let $S = \langle s \rangle$ and $T = \langle t \rangle$, and let $E$ be the edge between $S$ and $T$ in $\widehat \Phi_\Lambda$. 
    First, suppose $x$ and $y$ are vertices. 
    By symmetry and translating, we may assume $x = S$. 
    Let $\gamma = E_1E_2\cdots E_n$ be a geodesic edge path connecting $x$ and $y$. 
    Notice that $E_1$ and $E$ both contain $x = S$; this means there is some $k_1 \in \bb{Z}$ such that $E_1 = s^{k_1} E$ (we could have $k_1 = 0$ if $E_1 = E$). 
    Similarly, $E_2$ and $E_1$ both contain the vertex $s^{k_1}T$, so there is some $k_2 \in \bb{Z}$ such that $E_2 = s^{k_1} t^{k_2} E_1$. Since $\gamma$ is a geodesic, it is locally embedded, so $k_2 \not= 0$. 
    Repeating this, for all $i > 1$, we see $E_{i} = \alpha_i E_{i-1}$, where
    \[\alpha_i = \underbrace{s^{k_1}t^{k_2}s^{k_3} \cdots}_{i \text{ syllables}}, \] 
    for some $k_i \not= 0$. (A syllable is a word of the form $s^j$ or $t^j$.)
    
    Let $\alpha = \alpha_n$.
    Notice that
    \[
        \alpha \not\in C \coloneqq \underbrace{STS \cdots}_{n-1 \text{ cosets}},
    \]
    since otherwise we could reverse the above argument to construct a path from $x$ to $y$ with strictly shorter length than $\gamma$. 
    Since $A_\Lambda$ is product separable (Corollary \ref{cor:dihedralartinsep}), $C$ is closed in the profinite topology on $A_\Lambda$, so we can find a finite group $F$ and surjective morphism $\varphi : A_\Lambda \to F$ so that $\varphi(\alpha) \not\in \varphi(C)$ (Lemma \ref{lem:prodsepimpliesfinitequotient}).
    Let $p_s$ and $p_t$ be the orders of $\varphi(s)$ and $\varphi(t)$, respectively, and let $N = \lcm\{p_s,p_t\}$.
    Fix $k \geq 1$, and let $\overline{\,\cdot\,} : A_\Lambda \to \sh_{\Lambda(kN)}$ denote the standard quotient map. 
    The orders of $\varphi(s)$ and $\varphi(t)$ divide the orders of $\overline s$ and $\overline t$, resp., or in other words, $s^{kN}, t^{kN} \in \ker(\varphi)$. Since $\ker( \overline{\,\cdot\,} )$ is the normal closure of $s^{kN}$ and $t^{kN}$ in $A_\Lambda$, it follows that $\ker( \overline{\,\cdot\,} ) \leq \ker(\varphi)$. This means there exists a surjection $\rho : \sh_{\Lambda(kN)} \to F$ which makes the diagram
    \[\begin{tikzcd}[cramped,column sep=tiny]
        {A_\Lambda} && {F} \\
        & {\mathrm{Sh}_{\Lambda(kN)}}
        \arrow["{\varphi}", from=1-1, to=1-3]
        \arrow["{ \overline{\,\cdot\,} }"', from=1-1, to=2-2]
        \arrow["{\rho}"', dashed, from=2-2, to=1-3]
    \end{tikzcd}\]
    commute.
    This implies in particular that $\overline \alpha \not\in \overline C$; 
    otherwise
    $\varphi(\alpha) = \rho( \overline \alpha) \in \rho(\overline C) = \varphi(C)$. %
    
    By abuse of notation, let $\overline{\,\cdot\,} : \widehat \Phi_\Lambda \to \widehat \Theta_{\Lambda(kN)}$ denote the map induced by the quotient $A_\Lambda \to \sh_{\Lambda(kN)}$. 
    We claim that $\overline \gamma$ is still a geodesic between $\overline x$ and $\overline y$. 
    Let $\gamma'$ be a geodesic from $\overline x$ to $\overline y$ in $\widehat \Theta_{\Lambda(kN)}$, say $\gamma' = E_1'\cdots E_m'$. 
    By the same reasoning in $\widehat \Phi_\Lambda$,  for each $i$ we can find $\ell_i \in \bb{Z}$ so that $E_i' = \alpha_i' E_{i-1}'$, where 
    \[\alpha_i' = \underbrace{\overline s^{\ell_1}\overline t^{\ell_2}\overline s^{\ell_3} \cdots}_{i \text{ syllables}}, \] 
    and $\ell_i \not= 0$ when $i > 1$. 
    Let $\alpha' = \alpha_m'$.  

    Notice $E_m'$ and $\overline E_n$ meet at $\overline y$. 
    Since $\widehat\Theta_{\Lambda(kN)}$ is bipartite, we know $m$ and $n$ have the same parity. If they are odd, let $h = t$ and $H = T$, and if they are even, let $h = s$ and $H = S$. In either case, $\overline y$ is a coset of $\overline H$, so there is some $\ell_{m+1} \in \bb{Z}$ such that $\overline \alpha = \alpha' \overline h^{\ell_{m+1}}$. Writing out $\alpha'$, we see
    \[
        \overline \alpha
        =( \underbrace{\overline s^{\ell_1}\overline t^{\ell_2} \overline s^{\ell_3}\cdots}_{m \text{ syllables}} )\overline h^{\ell_{m+1}} 
        = \underbrace{\overline s^{\ell_1}\overline t^{\ell_2} \overline s^{\ell_3}\cdots}_{m+1 \text{ syllables}}.
    \]
    Since $\overline \alpha \not\in \overline C = \overline S\,\overline T\,\overline S \cdots$ ($n-1$ cosets), we know $m+1 > n-1$, i.e., $m > n-2$. But since $m$ and $n$ are the same parity, $m \geq n$. Thus $\overline \gamma$ is a geodesic and $d(\overline x, \overline y) = \ell(\overline \gamma) = (\pi/q)n = \ell(\gamma) = d(x,y)$. 
    
    Now suppose $x$ and $y$ are any two points of $\widehat \Phi_\Lambda$. If $x$ is a vertex, let $X_1 =X_2 = x$, and if $x$ is not a vertex, let $X_1$ and $X_2$ be the two distinct vertices of the edge containing $x$. Similarly, if $y$ is a vertex, let $Y_1 = Y_2 = y$, and if $y$ is not a vertex, let $Y_1$ and $Y_2$ be the two distinct vertices of the edge containing $y$. For $i,j \in\{ 1,2\}$,
    let $\gamma_{ij}$ be a geodesic from $X_i$ to $Y_j$. 
    For each $i,j$, by our work above we can find $N_{ij} \geq 1$ so that $\gamma_{ij}$ remains a geodesic of the same length under the quotient to $\widehat \Theta_{\Lambda(kN_{ij})}$ for each $k \geq 1$.
    Let $N = \lcm\{N_{ij}\}$ and fix $k \geq 1$. Let $\overline{\,\cdot\,} : \widehat\Phi_\Lambda \to \widehat\Theta_{\Lambda(kN)}$ be the quotient map. 
    Then each $\overline{\gamma_{ij}}$ is a geodesic in $\widehat \Theta_{\Lambda(kN)}$ of the same length as $\gamma_{ij}$.
    Suppose $\gamma$ is a geodesic from $\overline x$ to $\overline y$ in $\widehat\Theta_{\Lambda(kN)}$. 
    Then there is some $I$ and some $J$ so that $\gamma$ passes through $\overline {X_I}$ and $\overline {Y_J}$. 
    Note that since $\overline{\,\cdot\,}$ is a simplicial map, $d(\overline x, \overline{X_I}) = d(x, X_I)$ and $d(\overline y, \overline{Y_J}) = d(y,Y_J)$. 
    Then
    \begin{align*}
        d(\overline x, \overline y) 
        &= \ell(\gamma) \\
        &= \ell(\gamma|_{[\overline x, \overline X_I]})
         + \ell(\gamma|_{[\overline X_I, \overline Y_J]})
         + \ell(\gamma|_{[\overline Y_J, \overline y]}) \\
        &= d(\overline x, \overline X_I) + \ell(\gamma|_{[\overline X_I, \overline Y_J]}) + d(\overline Y_J, \overline y) \\
        &= d(\overline x, \overline X_I) + \ell(\overline \gamma_{IJ}) + d(\overline Y_J, \overline y) \\
        &= d(x,X_I) + \ell(\gamma_{IJ}) + d(Y_J, y) \\
        &= d(x,X_I) + d(X_I,Y_J) + d(Y_J, y) \\
        &\geq d(x,y).
    \end{align*}
    Since the quotient clearly cannot increase distance, $d(x,y) = d(\overline x, \overline y)$. 
\end{proof}

\begin{lemma} \label{lem:keyartresid}
    Suppose $\Gamma$ is a 2-dimensional presentation graph.
    Let $\gamma$ be a geodesic segment of $\Phi_\Gamma$. 
    Then there is some $k \geq 2$ so that the image of $\gamma$ under the natural quotient $\Phi_\Gamma \to \Theta_{\Gamma(k)}$ remains a geodesic segment of the same length.
\end{lemma}

\begin{proof}
    For $k \geq 2$, let $\rho_k : \Phi_\Gamma \to \Theta_{\Gamma(k)}$ denote the usual quotient map induced by the quotient $q_k : A_\Gamma \to \sh_{\Gamma(k)}$. 
    Let $x$ be a point in the interior of $\gamma$. 
    
    If $x$ is in the interior of a 2-simplex of $\Phi_\Gamma$, then $\rho_k(\gamma)$ is locally geodesic at $\rho_k(x)$ for any value of $k$, since $\rho_p$ maps 2-simplices isometrically to 2-simplices. In this case, define $n_x = 2$.
    
    If there is a neighborhood of $x$ in $\gamma$ which is contained in an edge, then there is nothing to show since edges are mapped isometrically to edges (in other words, $\rho_k(\gamma)$ is locally geodesic at $\rho_k(x)$ for any value of $k$). In this case, as above, define $n_x = 2$.
    
    Suppose $x$ is a point contained in the interior of an edge $E$ of $\Phi_\Gamma$ with no neighborhood of $x$ in $\gamma$ also contained in this edge. 
    First, suppose $E = [a_1A_\varnothing, a_2A_s]$ for a vertex $s$ of $\Gamma$ or $E = [a_1A_\varnothing, a_2A_e]$ for an edge $e$ of $\Gamma$. Then the stabilizer of $E$ is trivial, and in particular, the link of $E$ is mapped isomorphically to the link of $\rho_k(E)$ for any $k$.
    So suppose
    $E = [a_1 A_s, a_2 A_e]$ for a vertex $s$ of $\Gamma$ and an edge $e$ of $\Gamma$ containing $s$. 
    Let $\Delta_1$ and $\Delta_2$ denote the distinct 2-simplices containing $E$ whose interiors nontrivially intersect $\Gamma$.
    The stabilizer of $E$ is conjugate to $\langle s \rangle$ and acts transitively on the link of $E$, so there is some $g \in A_\Gamma$ and some $n_x \in \bb{Z}$ such that $\Delta_2 = (g^{-1} s^{n_x} g) \Delta_1$.  
    (Since $\gamma$ is a geodesic, $n_x \not= 0$.) For $k > |n_x|$,
    $\rho_k(\Delta_1)$ and $\rho_k(\Delta_2)$ are the 2-simplices of $\Theta_{\Gamma(k)}$ containing $\rho_k(E)$ (and $\rho_k(x)$) whose interiors intersect $\rho_k(\gamma)$. 
    Moreover, $\rho_k(\Delta_2) = q_k(g^{-1} s^{n_x} g) \rho_k(\Delta_1)$. 
    If $\rho_k(\Delta_2) = \rho_k(\Delta_1)$, then $q_k(g^{-1} s^{n_x} g) = e$, but this happens if and only if $q_k(s^{n_x}) = e$. 
    Since $k > |n_x|$, this can't happen, so $\rho_k(\Delta_2) \not= \rho_k(\Delta_1)$ and $\rho_k(\gamma)$ is locally geodesic at $\rho_k(x)$ for all $k > |n_x|$.

    Last, suppose $x$ is a vertex of $\Phi_\Gamma$. Then $x = [g A_\Lambda]$ for some $g \in A_\Gamma$ and some $\Lambda \in \c S^f$. There are three subcases to consider.
    
    First, assume $\Lambda = \varnothing$. Then for any $k \geq 2$, $\rho_k(x) = [q_k(g) \sh_\varnothing]$, and $\rho_k$ induces an isometry of the link of $x$ and the link of $\rho_k(x)$ (both being isometric to $lk(v_\varnothing, F_\varnothing)$).
    
    Next, assume $\Lambda = \{s\}$ for a vertex $s$ of $\Gamma$. The link of $x$ is the join of $lk(v_s,F_s)$ and $\widehat \Phi_{\{s\}}$, the latter of which is in bijection with $\langle s \rangle \cong \bb{Z}$. 
    The intersection of the $\varepsilon$-sphere of $x$ with $\gamma$ induces two points $y,z \in lk(x)$ of distance $\geq \pi$ apart. 
    Every point of $\widehat \Phi_{\{s\}}$ is distance $\pi/2$ from every point of $lk(v_s,F_s)$ (by the definition of spherical join), so either $y,z \in lk(v_s,F_s)$ or $y,z \in \widehat \Phi_{\{s\}}$%
    .
    If $y,z \in lk(v_s,F_s)$, then as in the previous case $\rho_k$ acts isometrically on $lk(v_s,F_s)$ for any $k\geq 2$, so $y$ and $z$ remain at the same distance under $\rho_k$. 
    If $y,z \in \widehat \Phi_{\{s\}}$, then they can be represented by some powers of $s$, say $s^n$ and $s^m$ with $n \not= m \in \bb{Z}$.
    (Since $\gamma$ is a geodesic, these points are distinct.) 
    By translating we may assume these points are in fact $e$ ($= s^0$) and $s^{n_x}$ for $n_x = m - n \not= 0$. 
    Choosing $k > |n_x|$,  $\rho_k$ acts by the standard quotient $\bb{Z} \to \bb{Z}/k\bb{Z}$ and thus these points remain distinct. In other words, for $k > |n_x|$, $\rho_k(\gamma)$ is locally geodesic at $\rho_k(x)$. 
    
    Finally, suppose $\Lambda$ is an edge with finite label $m$. Then, as before, the intersection of $\gamma$ with the $\varepsilon$-sphere of $x$ induces two points $y,z \in lk(x) = \widehat \Phi_\Lambda$ of distance $\geq \pi$. Lemma \ref{lem:dihedralartinquotient} implies that there is some $n_x \geq 1$ so that the distance between $y$ and $z$ remains $\geq \pi$ in the quotient via $\rho_k$ for all $k$ which are (positive) multiples of $n_x$. In other words, for all such $k$, $\rho_k(\gamma)$ is locally geodesic at $\rho_k(x)$. 
    
    Let $k \geq 2$ be a common multiple of $\{\,|n_x| : x \in \mathrm{int}(\gamma) \,\}$. (This set is finite and consists of positive integers, so $k$ exists.) Our above arguments show that $\rho_k(\gamma)$ is locally geodesic at $\rho_k(x)$ for all $x$, or in other words, that $\rho_k(\gamma)$ is a local geodesic in $\Theta_{\Gamma(k)}$. Since $\Theta_{\Gamma(k)}$ is $\cat(0)$, this implies $\rho_k(\gamma)$ is a geodesic and has the same length as $\gamma$.
\end{proof}

Now we may complete the proof of %

\artinresiduallyfinite*

\begin{proof} %
    Let $g \in A_\Gamma \setminus \{e\}$. Let $\gamma$ be the geodesic in $\Phi_\Gamma$ from $[A_\varnothing]$ to $[gA_\varnothing]$. By Lemma \ref{lem:keyartresid}, there is some $k$ so that the image $\overline \gamma$ of $\gamma$  under the quotient map to $\Theta_{\Gamma(k)}$ is a geodesic. If $\overline g$ denotes the image of $g$ under the quotient $A_\Gamma \to \sh_{\Gamma(k)}$, then $\overline \gamma$ is a geodesic from $[\sh_\varnothing]$ to $[\overline g \sh_\varnothing]$. Since $\Theta_{\Gamma(k)}$ is $\cat(0)$, this means $\overline g \not= e$. But $\sh_{\Gamma(k)}$ is residually finite (Corollary \ref{cor:residuallyfinite}), so there is a further quotient $\sh_{\Gamma(k)} \to F$ to a finite group $F$ under which $\overline g$ remains nontrivial. Composing these maps gives a quotient $A_\Gamma \to F$ to a finite group where the image of $g$ is nontrivial.
\end{proof}

\bibliographystyle{alpha}
\bibliography{2dshep}

\end{document}